\documentclass[sn-mathphys-num]{sn-jnl}

\usepackage{amsmath}
\usepackage{bm}
\usepackage{graphicx}%
\usepackage{multirow}%
\usepackage{amsmath,amssymb,amsfonts}%
\usepackage{amsthm}%
\usepackage{mathrsfs}%
\usepackage[title]{appendix}%
\usepackage{xcolor}%
\usepackage{textcomp}%
\usepackage{manyfoot}%
\usepackage{booktabs}%
\usepackage{algorithm}%
\usepackage{algorithmicx}%
\usepackage{algpseudocode}%
\usepackage{listings}%

\theoremstyle{thmstyleone}%
\newtheorem{theorem}{Theorem}
\newtheorem{lemma}{Lemma}
\theoremstyle{thmstyletwo}%
\newtheorem{remark}{Remark}%

\theoremstyle{thmstylethree}%

\newcommand{\om}{\Omega}
\def\C{\mathbb{C}}
\newcommand{\p}{\partial}
\def\d{\mathrm{d}}

\newcommand{\ve}{\varepsilon}
\newcommand{\ds}{\displaystyle}
\newcommand{\f}{\frac}

\begin{document}

\title[Variable exponent fractional problems]{Two methods addressing variable-exponent fractional initial and boundary value problems and Abel integral equation}


\author[1]{\fnm{Xiangcheng} \sur{Zheng}}\email{xzheng@sdu.edu.cn}

\affil[1]{\orgdiv{School of Mathematics}, \orgname{Shandong University}, \orgaddress{\city{Jinan}, \postcode{250100}, \country{China}}}


\abstract{Variable-exponent fractional models attract increasing attentions in various applications, while the rigorous analysis is far from well developed. This work provides general tools to address these models. Specifically, we first develop a convolution method to study the well-posedness, regularity, an inverse problem and numerical approximation for the sundiffusion of variable exponent. For models such as the variable-exponent two-sided space-fractional boundary value problem (including the variable-exponent fractional Laplacian equation as a special case) and the distributed variable-exponent model, for which the convolution method does not apply, we develop a perturbation method to prove their well-posedness. The relation between the convolution method and the perturbation method is discussed, and we further apply the latter to prove the well-posedness of the variable-exponent Abel integral equation and discuss the constraint on the data under different initial values of variable exponent. }




\maketitle

\section{Introduction}
Variable-exponent fractional models attract increasing attentions in various fields. For instance, the variable-exponent subdiffusion equation has been widely used in the anomalous transient dispersion \cite{SunZha,SunChe}, while the variable-exponent two-sided space-fractional diffusion equation has been applied in, e.g., the turbulent channel flow \cite{SonKar} and the transport through a highly 
heterogeneous medium \cite{PanPer,ZenZhaKar}. More applications of variable-exponent fractional models can be found in a comprehensive review  \cite{SunCha}.

Nevertheless, the theoretical study of variable-exponent fractional initial and boundary value problems as well as the corresponding integral equations is far from well developed. For the subdiffusion model, a typical fractional initial value problem, there exist extensive investigations and significant progresses for the constant-exponent case \cite{Ban,Giga,Giga2,KiaLiu,Kop,Kub,LiZha,LiaTan,LinNak,Luc,ZacMA,Zac19}, while much fewer investigations for the variable exponent case can be found in the literature.  There exist some mathematical and numerical results for the case of space-variable-dependent variable exponent \cite{KiaSocYam,LeSty,Zha}. For the time-variable-dependent case, there are some recent works changing the exponent in the Laplace domain \cite{Cue,GarGiu} such that the Laplace transform of the variable-exponent fractional operator is available. For the case that the exponent changes in the time domain, the only available results focus on the piecewise-constant variable exponent case such that the solution representation is available on each temporal piece \cite{KiaSlo,Uma}. It is also commented in \cite{KiaSlo} that the case of smooth exponent remains an open problem. Recently, a local modification of subdiffusion by variable exponent $\alpha(t)$ is proposed in \cite{ZheLi}, where the techniques work only for the case $\alpha(0)=0$ such that this is mathematically a special case of the variable-exponent subdiffusion model. The general case $0<\alpha(0)<1$ remains untreated.

For fractional and nonlocal boundary value problems, there also exist sophisticated investigations for the constant-exponent case, see e.g. \cite{Aco,Del,Dubook,DuSR,ErvJDE,Erv,JinLaz}. For the variable-exponent case, there are some corresponding theoretical results such as heat kernel estimates for variable-exponent nonlocal operators \cite{Che} and well-posedness study for a nonlocal model involving the doubly-variable  fractional exponent and possibly truncated  interaction \cite{DelGlu}. Numerically, some efficient algorithms have been developed for the variable-exponent nonlocal and fractional Laplacian operators \cite{DelGlu,WuZha}. Furthermore, a solution landscape approach has been applied for nonlinear problems involving variable-exponent spectral fractional Laplacian \cite{YuZhe}. Nevertheless, rigorous analysis of variable-exponent fractional boundary value problems defined via fractional derivatives, e.g., the variable-exponent two-sided space-fractional diffusion equation in the form as \cite{Erv}, is not available in the literature.

For weakly singular integral equations such as the Abel integral equation, extensive results for the constant-exponent models can be found in the literature, see e.g. the books \cite{Bru,Gorbook}. For the case of variable exponent, there are rare studies. Recently, some works consider the mathematical and numerical analysis for the second-kind weakly singular Volterra intrgral equation of the variable exponent \cite{LiaSty,ZheWanSINUMa}. For the first-kind integral equations, \cite{ZheFCAA} develops an approximate inverse technique to convert the non-convolutional Abel intrgral equation of variable exponent to a second-kind weakly singular Volterra intrgral equation to facilitate the analysis. For the variable-exponent Abel intrgral equation of convolutional form, which is a more natural way to introduce the variable exponent, the corresponding study is not available.

The main difficulty for investigating variable-exponent fractional initial or boundary value problems as well as the first-kind variable-exponent Volterra integral equations is that the variable-exponent Abel kernel in these models could not be analytically treated by, e.g. the integral transform, and may not be positive definite or monotonic such that existing methods for partial differential equations do not apply directly. It is worth mentioning that in some variable-exponent fractional models considered in, e.g. \cite{ZheWanSINUMa,ZheWanSINUMb}, both integer-order and variable-exponent fractional terms exist at the same time in space or time such that the latter serve as low-order terms.  For this case, the complexity of the variable-exponent operators is weakened such that the analysis could be performed. For models considered in the current work, where the variable-exponent fractional terms serve as leading terms, the methods in \cite{ZheWanSINUMa,ZheWanSINUMb} do not apply.

In this work, we develop two methods, i.e. the {\it convolution method} and the {\it  perturbation method}, to address the aforementioned issues, which provide general tools for analyzing different kinds of variable-exponent nonlocal models, including variable-exponent initial and boundary value problems and the variable-exponent Abel integral equation. Specifically, the main contributions are enumerated as follows:
\begin{itemize}
\item[$\blacktriangleright$] \textbf{Convolution method for initial value problems}\\
  We develop a convolution method to treat the variable-exponent fractional initial value problems, the idea of which is to calculate the convolution of the original model and a suitable kernel to get a more feasible formulation such that the analysis could be considered.  The key of this method lies in introducing a ``nice'' kernel such that its convolution with the variable-exponent kernel leads to a so-called generalized identity function, which has favorable properties. Thus, the variable-exponent term could be transformed to a more feasible form without deteriorating the properties of the other terms in the model (as the introduced kernel is nice) such that the transformed model becomes tractable. 

We apply the convolution method to give a mathematical and numerical analysis for subdiffusion of variable exponent, a typical fractional initial value problem, including the following contents:
\begin{itemize}
\item[$\bullet$] We prove its well-posedness and regularity by means of resolvent estimates, and characterize the singularity of the solutions in terms of the initial value $\alpha(0)$ of the exponent, which indicates an intrinsic factor of variable-exponent models. 

\item[$\bullet$] The semi-discrete and fully-discrete numerical methods are developed and their error estimates are proved, without any regularity assumption on solutions or requiring specific properties of the variable-exponent Abel kernel such as positive definiteness or monotonicity. We again characterize the convergence order by $\alpha(0)$. 

\item[$\bullet$] Since the $\alpha(0)$ plays a critical role in characterizing properties of the model and the numerical method, we present a preliminary result for the inverse problem of determining $\alpha(0)$. 
\end{itemize}

\item[$\blacktriangleright$]  \textbf{Perturbation method for boundary value problems and integral equations}\\
 For problems such as variable-exponent two-sided space-fractional boundary value problems or distributed variable-exponent problems, the convolution method do not apply, cf. Section \ref{secc51} for detailed reasons. Thus, we develop an alternative method called the perturbation method to treat these models, the idea of which is to replace the variable-exponent kernel by a suitable ``nice'' kernel such that their difference is a low-order perturbation term. Note that different from the convolution method, the perturbation method treats the variable-exponent term locally such that all other terms in the model keep unchanged. 

Then we investigate and apply the perturbation method from the following aspects:
\begin{itemize}
\item[$\bullet$] We first show that for some models such as the variable-exponent subdiffusion, the convolution method and the perturbation method are indeed the same method in the sense that both methods could lead to the same transformed model.

\item[$\bullet$] We apply the  perturbation method to prove the well-posedness of variable-exponent two-sided space-fractional diffusion-advection-reaction equation (including the variable-exponent fractional Laplacian equation as a special case) by showing the coercivity and continuity of the corresponding bilinear form, which are difficult to prove from the original form of the model since the variable-exponent fractional operators are not positive definite or monotonic.

\item[$\bullet$] We apply the perturbation method to prove the well-posedness of the weak solution of a distributed variable-exponent model.

\item[$\bullet$] We further apply the perturbation method to analyze the well-posedness of the variable-exponent Abel integral equation for three cases: $0<\alpha(0)<1$, $\alpha(0)=0$ and $\alpha(0)=1$. In particular, we observe that for the last two cases, an additional constraint is required to ensure the well-posedness, and a further extension implies that the well-posedness of the variable-exponent subdiffusion with $\alpha(0)=1$ also requires an additional constraint on initial values of the data.
\end{itemize}

\end{itemize}

The rest of the work is organized as follows: In Section 2 we introduce notations and preliminary results. In Section 3 we introduce the convolution method and then apply it to prove the well-posedness and regularity of the subdiffusion of variable exponent, as well as an inverse problem of determining the $\alpha(0)$. In Section 4,  semi-discrete and fully-discrete schemes for subdiffusion of variable exponent are proposed and analyzed. In Section 5, we introduce the perturbation method and then apply it to analyze the variable-exponent two-sided space-fractional boundary value problem and the distributed variable-exponent model. In Section 6 we further apply the perturbation method to analyze the variable-exponent Abel integral equation.

\section{Preliminaries}\label{sec2}
\subsection{Notations}
Let $L^p(\om)$ with $1 \le p \le \infty$ be the Banach space of $p$th power Lebesgue integrable functions on $\om$. Denote the inner product of $L^2(\Omega)$ as $(\cdot,\cdot)$. For a positive integer $m$,
let  $ W^{m, p}(\Omega)$ be the Sobolev space of $L^p$ functions with $m$th weakly derivatives in $L^p(\om)$ (similarly defined with $\om$ replaced by an interval $\mathcal I$). Let  $H^m(\Omega) := W^{m,2}(\Omega)$ and $H^m_0(\Omega)$ be its subspace with the zero boundary condition up to order $m-1$. For $\frac{1}{2}<s\leq 1$, $H^s_0(\Omega)$ is defined via interpolation of $L^2(\Omega)$ and $H^1_0(\Omega)$ \cite{AdaFou}. 

Let $\{\lambda_i,\phi_i\}_{i=1}^\infty$ be eigen-pairs of the problem $-\Delta \phi_i = \lambda_i \phi_i$ with the zero boundary condition. We introduce the Sobolev space $\check{H}^s(\Omega)$ for $s\geq 0$ by
$ \check{H}^{s}(\Omega) := \big \{ q \in L^2(\Omega): \| q \|_{\check{H}^s}^2 : = \sum_{i=1}^{\infty} \lambda_i^{s} (q,\phi_i)^2 < \infty \big \},$
which is a subspace of $H^s(\Omega)$ satisfying $\check{H}^0(\Omega) = L^2(\Omega)$ and $\check{H}^2(\Omega) = H^2(\Omega) \cap H^1_0(\Omega)$ \cite{Tho}.
For a Banach space $\mathcal{X}$, let $W^{m, p}(0,T; \mathcal{X})$ be the space of functions in $W^{m, p}(0,T)$ with respect to $\|\cdot\|_{\mathcal {X}}$.  All spaces are equipped with standard norms \cite{AdaFou,Eva}.

Before introducing fractional derivative spaces, we present fractional operators. For $0<\mu<1$, define 
$$\beta_\mu (t):=\frac{t^{\mu-1}}{\Gamma(\mu)},$$
 the left-sided fractional integral operator
$$I_t^\mu q:=\beta_\mu*q=\int_0^t \beta_\mu(t-s)q(s)ds,$$
 the left-sided Riemann-Liouville fractional derivative operator $\p_t^\mu q:=\p_t (\beta_{1-\mu}*q)$ and the left-sided Caputo fractional derivative operator $^c\p_t^\mu q:=\beta_{1-\mu}*(\p_tq)$ \cite{Jinbook}. The right-sided operators are denoted as $\hat I_t^\mu$, $\hat \p_t^\mu$ and $^c\hat \p_t^\mu$, respectively, with standard definitions \cite{Jinbook}. On a finite interval $\Omega$, the following semigroup property and the adjoint property hold
\begin{align}\label{er0}
I_x^{\mu_1}I_x^{\mu_2}q=I_x^{\mu_1+\mu_2}q,~~\mu_1,\mu_2\geq 0;~~(I_x^\mu q,\bar q)=(q,\hat I_x^\mu\bar q),
\end{align}
 and it is proved in \cite{Erv} that the following equivalent formulas hold for $u,v\in H^\mu_0(\Omega)$ with $0<\mu<\frac{1}{2}$ or $\frac{1}{2}<\mu<1$ and for some positive constants $c_1$--$c_4$
\begin{align}
&c_1\|u\|^2_{H^\mu(\Omega)}\leq (-1)^{\sigma}(\p_x^\mu u,\hat \p_x^\mu u)\leq c_2 \|u\|^2_{H^\mu(\Omega)},\nonumber\\
&~~c_3\|\p_x^\mu u\|_{L^2(\Omega)}\leq \|u\|_{H^\mu(\Omega)}\leq c_4\|\p_x^\mu u\|_{L^2(\Omega)},\label{er}\\
&~~c_3\|\hat \p_x^\mu u\|_{L^2(\Omega)}\leq \|u\|_{H^\mu(\Omega)}\leq c_4\|\hat\p_x^\mu u\|_{L^2(\Omega)},\nonumber
\end{align}
where $\sigma=0$ for $0<\mu<\frac{1}{2}$ and $\sigma=1$ for 
$\frac{1}{2}<\mu<1$.
 
We use $Q$ to denote a generic positive constant that may assume different values at different occurrences. We use $\|\cdot\|$ to denote the $L^2$ norm of functions or operators in $L^2(\Omega)$, set $L^p(\mathcal X)$ for $L^p(0,T;\mathcal X)$ for brevity, and drop the notation $\om$ in the spaces and norms if no confusion occurs. For instance, $L^2(L^2)$ implies $L^2(0,T;L^2(\Omega))$. Furthermore, we will drop the space variable $\bm x$ in functions, e.g. we denote $q(\bm x,t)$ as $q(t)$, when no confusion occurs.

For $\theta\in(\pi/2,\pi)$ and $\delta > 0$, let $\Gamma_\theta$ be the contour in the complex plane defined by
$
\Gamma_\theta := \big \{z\in\C: |{\rm arg}(z)|=\theta, |z|\ge \delta \big \}
\cup \big \{z \in\C: |{\rm arg}(z)|\le \theta, |z|= \delta \big \}.
$
For any $q \in L^1_{loc}(\mathcal I)$, the Laplace transform $\mathcal L$ of its extension $\tilde q(t)$ to zero outside $\mathcal I$ and the corresponding inverse transform $\mathcal L^{-1}$ are denoted by
\begin{equation*}
\ds \mathcal{L}q(z):=\int_0^\infty \tilde q(t)e^{-tz}d t, \quad \mathcal{L}^{-1}(\mathcal Lq(z)):=\frac{1}{2\pi \rm i}\int_{\Gamma_\theta} e^{tz}\mathcal Lq(z)d z=q(t).
\end{equation*}
The following inequalities hold for  $0<\mu< 1$ and $Q=Q(\theta,\mu)$ \cite{Akr,Lub,ShiChe}
\begin{equation}\label{GammaEstimate}
\int_{\Gamma_\theta} |z|^{\mu-1} |e^{tz}|  \, |d z| \le Q t^{-\mu},
~~t\in (0,T];~~\|(z^\mu-\Delta)^{-1}\|\leq Q|z|^{-\mu},~~\forall z\in \Gamma_\theta,
\end{equation}
and we have $\mathcal L(\p_t^\mu q)=z^\mu \mathcal L q-(I_t^{1-\mu}q)(0) $ \cite{Jinbook}.

In Sections \ref{sec2}--\ref{sec4}, we consider smooth variable exponent $0<\alpha(t)<1$ on $t\in [0,T]$ or $t\in [0,1]$ such that $\alpha(t)$ is three times differentiable with $|\alpha'(t)|+|\alpha''(t)|+|\alpha'''(t)|\leq L$ for some $L>0$. In the last section, the special cases $\alpha(0)=0$ or $\alpha(0)=1$ (with $0<\alpha(t)<1$ on $t\in (0,T]$) will be considered. Furthermore, we denote $\alpha_0=\alpha(0)$ and  $\alpha_1=\alpha(1)$ for simplicity.

\subsection{A generalized identity function}
 We propose the concept of the generalized identity function. For the variable-exponent Abel kernel
 $$k(t):=\frac{t^{-\alpha(t)}}{\Gamma(1-\alpha(t))}, $$
 direct calculations show that 
\begin{align}
(\beta_{\alpha_0}*k)(t)=&\int_0^t\frac{(t-s)^{\alpha_0-1}}{\Gamma(\alpha_0)}\frac{s^{-\alpha(s)}}{\Gamma(1-\alpha(s))}ds\nonumber\\
&\overset{z=s/t}{=}\int_0^1\frac{(t-tz)^{\alpha_0-1}}{\Gamma(\alpha_0)}\frac{(tz)^{-\alpha(tz)}}{\Gamma(1-\alpha(tz))}tdz\label{mh7}\\
&=\int_0^1\frac{(tz)^{\alpha_0-\alpha(tz)}}{\Gamma(\alpha_0)\Gamma(1-\alpha(tz))}(1-z)^{\alpha_0-1}z^{-\alpha_0}dz=:g(t).\nonumber
\end{align}
It is clear that if $\alpha(t)\equiv \alpha$ for some constant $0<\alpha<1$, then 
$$
g(t)=\int_0^1\frac{1}{\Gamma(\alpha)\Gamma(1-\alpha)}(1-z)^{\alpha-1}z^{-\alpha}dz=1. 
$$
For the variable exponent $\alpha(t)$, $g(t)\not\equiv 1$ in general. However, as 
$$\lim_{t\rightarrow 0^+} |(\alpha_0-\alpha(tz))\ln tz|\leq \lim_{t\rightarrow 0^+}Ltz|\ln tz|=0,$$
 we have
$\lim_{t\rightarrow 0^+}(tz)^{\alpha_0-\alpha(tz)}=\lim_{t\rightarrow 0^+}e^{(\alpha_0-\alpha(tz))\ln tz}=1, $
which implies
\begin{align}\label{g0}
g(0)= \int_0^1\frac{1}{\Gamma(\alpha_0)\Gamma(1-\alpha_0)}(1-z)^{\alpha_0-1}z^{-\alpha_0}dz=1.
\end{align}
Based on this property, we call $g(t)$ a {\it generalized identity function}, which will play a key role in reformulating the model (\ref{VtFDEs}) to more feasible forms for mathematical and numerical analysis.

It is clear that $g$ is bounded over $[0,T]$. To bound derivatives of $g$, we use $(tz)^{\alpha_0-\alpha(tz)}\leq Q$ to obtain
 \begin{align*}
 \Big|\p_t(tz)^{\alpha_0-\alpha(tz)}\Big|&=\Big|(tz)^{\alpha_0-\alpha(tz)}z\Big(-\alpha'(tz)\ln(tz)+\frac{\alpha_0-\alpha(tz)}{tz}\Big) \Big|\\
 &\leq Qz\big(|\ln(tz)|+1\big),\\
 \Big|\p_t^2(tz)^{\alpha_0-\alpha(tz)}\Big|&=\Big|\big[\p_t(tz)^{\alpha_0-\alpha(tz)}\big]z\Big(-\alpha'(tz)\ln(tz)+\frac{\alpha_0-\alpha(tz)}{tz}\Big) \\
 &\quad+(tz)^{\alpha_0-\alpha(tz)}z^2\Big(-\alpha''(tz)\ln(tz)-2\frac{\alpha'(tz)}{tz}-\frac{\alpha_0-\alpha(tz)}{(tz)^2}\Big) \Big|\\
 &\leq Qz^2\big(|\ln(tz)|+1\big)^2+z^2(|\ln(tz)|+\frac{1}{tz})\leq Q\frac{z^2}{tz}=Q\frac{z}{t},\\
 \Big|\p_t^3(tz)^{\alpha_0-\alpha(tz)}\Big|&\leq \frac{Q}{t^2} ~(\text{We omit the calculations here}).
 \end{align*}
We apply them to bound $g'(t)$, $g''(t)$ and $g'''(t)$ as
\begin{align}
|g'|&=\Big|\int_0^1\p_t\Big(\frac{(tz)^{\alpha_0-\alpha(tz)}}{\Gamma(\alpha_0)\Gamma(1-\alpha(tz))}\Big)(1-z)^{\alpha_0-1}z^{-\alpha_0}dz\Big| \nonumber\\
&\leq Q\int_0^1z\big(|\ln(tz)|+1\big)(1-z)^{\alpha_0-1}z^{-\alpha_0}dz\nonumber\\
&\hspace{-0.15in}\leq Q\big(|\ln t|+1\big)\int_0^1z\frac{|\ln t|+|\ln z|+1}{|\ln t|+1}(1-z)^{\alpha_0-1}z^{-\alpha_0}dz\leq Q\big(|\ln t|+1\big),\label{g'}\\
|g''|&=\Big|\int_0^1\p_t^2\Big(\frac{(tz)^{\alpha_0-\alpha(tz)}}{\Gamma(\alpha_0)\Gamma(1-\alpha(tz))}\Big)(1-z)^{\alpha_0-1}z^{-\alpha_0}dz\Big| \leq \frac{Q}{t},\label{g''}\\
|g'''|&=\Big|\int_0^1\p_t^3\Big(\frac{(tz)^{\alpha_0-\alpha(tz)}}{\Gamma(\alpha_0)\Gamma(1-\alpha(tz))}\Big)(1-z)^{\alpha_0-1}z^{-\alpha_0}dz\Big| \leq \frac{Q}{t^2}\label{g'''}.
\end{align}

\section{Convolution method}\label{seccm}

We develop the convolution method to analyze the fractional initial value problems of variable exponent, the main idea of which is to calculate the convolution of the original model and a suitable kernel to get a more feasible formulation. We illustrate this idea by the following subdiffusion model  with variable exponent, which is a typical fractional parabolic equation
\begin{equation}\label{VtFDEs}\begin{array}{c}
^c\p_t^{\alpha(t)} u(\bm x,t)-  \Delta u (\bm x,t)= f(\bm x,t) ,~~(\bm x,t) \in \Omega\times(0,T],
\end{array}\end{equation}
equipped with initial and boundary conditions
\begin{equation}\label{ibc}
 u(\bm x,0)=u_0(\bm x),~\bm x\in \Omega; \quad u(\bm x,t) = 0,~(\bm x,t) \in \p \Omega\times[0,T]. \end{equation}
Here $\Omega \subset \mathbb{R}^d$  is a simply-connected bounded domain with the piecewise smooth boundary $\p \om$ with convex corners, $\bm x := (x_1,\cdots,x_d)$ with $1 \le d \le 3$ denotes the spatial variables, $f$ and $u_0$ refer to the source term and the initial value, respectively, and the fractional derivative of order $0<\alpha(t)<1$ is defined via the variable-exponent Abel kernel $k$ \cite{LorHar}
\begin{align*}
^c\p_t^{\alpha(t)}u :=k*\p_t u .
\end{align*}

It is worth mentioning that although we focus the attention on the fractional parabolic equation (i.e. $0<\alpha(t)<1$), the proposed method also works for the fractional hyperbolic equation (i.e. $1<\alpha(t)<2$) based on resolvent estimates proposed in, e.g. \cite{ZhaZho}. The latter is also known as the fractional diffusion-wave equation, for which the $\p_t^{\alpha(t)}u$ is defined as
\begin{align*}
^c\p_t^{\alpha(t)}u :=\tilde k*\p_t^2 u,~~ \tilde k(t):=\frac{t^{1-\alpha(t)}}{\Gamma(2-\alpha(t))}.
\end{align*}

\subsection{A transformed model}
To demonstrate the idea of the convolution method, we calculate the convolution of (\ref{VtFDEs}) and $\beta_{\alpha_0}$ as
$$\beta_{\alpha_0}*[(k*\p_t u)-\Delta u - f]=0. $$
We invoke (\ref{mh7}) to obtain
\begin{align}\label{modeln}
g*\p_t u-\beta_{\alpha_0}*  \Delta u  -\beta_{\alpha_0}*f=0. 
\end{align}
We finally use (\ref{g0}) to formally differentiate it to get the transformed model
\begin{equation}\label{Model2}
\p_tu+g'*\p_t u-\p_t^{1-\alpha_0}  \Delta u=\p_t^{1-\alpha_0}f,
\end{equation}
equipped with the initial and boundary conditions (\ref{ibc}).

We then take the Laplace transform of (\ref{Model2}) to obtain
$$z\mathcal L u-u_0-z^{1-\alpha_0}\Delta\mathcal Lu=\mathcal L p_u,~~p_u(t):=\p_t^{1-\alpha_0}f-g'*\p_t u, $$
that is,
$\mathcal Lu=z^{\alpha_0-1}(z^{\alpha_0}-\Delta)^{-1}(u_0+\mathcal Lp_u). $
Take the inverse Laplace transform to get
\begin{align}\label{repu}
u=F(t)u_0+\int_0^tF(t-s)p_u(s)ds
\end{align}
where the operator $F(\cdot):L^2(\Omega)\rightarrow L^2(\Omega)$ is defined as
$$F(t)q:=\frac{1}{2\pi \rm i}\int_{\Gamma_\theta} e^{tz}z^{\alpha_0-1}(z^{\alpha_0}-\Delta)^{-1}qd z,~~\forall q\in L^2(\Omega), $$
with the following properties \cite{Jinbook}
\begin{align}\label{propF}
\|F(t)q\|_{\check H^{\mu}}\leq Qt^{\frac{\gamma-\mu}{2}\alpha_0}\|q\|_{\check H^\gamma}\text{ for }\mu,\gamma\in \mathbb R\text{ and }\gamma\leq \mu\leq \gamma+2.
\end{align}
\subsection{Well-posedness of transformed model}
We show the well-posedness of model (\ref{Model2}) with the initial and boundary conditions (\ref{ibc}) in the following theorem.
\begin{theorem}\label{thm:aux}
Suppose $f\in W^{1,1}(L^2)$ and $u_0\in \check H^2$, then model (\ref{Model2})--(\ref{ibc}) admits a unique solution in $W^{1,p}(L^2)\cap L^p(\check H^2)$ for $1< p<\frac{1}{1-\alpha_0}$ and
\begin{align*}
&\|u\|_{W^{1,p}(L^2)}+\| u\|_{L^p(\check H^2)}\leq Q\big(\|f\|_{W^{1,1}(L^2)}+\|u_0\|_{\check H^2}\big).
\end{align*}
\end{theorem}
\begin{proof}
By $C([0,T];L^2)\subset W^{1,1}(L^2)$ \cite{AdaFou}, we have the following relations for $0<\varepsilon\ll \frac{1}{p}-(1-\alpha_0)$
\begin{align}\label{imh}
I_t^{1-\varepsilon}\p_t^{1-\alpha_0}f=I_t^{\alpha_0-\varepsilon}f,~~\p_t^{\varepsilon}\p_t^{1-\alpha_0}f=\beta_{\alpha_0-\varepsilon}f(0)+\beta_{\alpha_0-\varepsilon}*\p_tf, 
\end{align}
which means that 
\begin{align}\label{regf}
(I_t^{1-\varepsilon}\p_t^{1-\alpha_0}f)|_{t=0}=0,~~\|\p_t^{\varepsilon}\p_t^{1-\alpha_0}f\|_{L^p(L^2)}\leq Q\|f\|_{W^{1,1}(L^2)}. 
\end{align}
Thus we take $0<\varepsilon\ll \frac{1}{p}-(1-\alpha_0)$ throughout this proof.

We first consider the case of zero initial condition, i.e. $u_0\equiv 0$, and define the space $\tilde W^{1,p}(L^2):=\{q\in W^{1,p}(L^2):q(0)=0\}$ equipped with the norm $\|q\|_{\tilde W^{1,p}(L^2)}:=\|e^{-\sigma t}\p_tq\|_{L^p(L^2)}$ for some $\sigma\geq 1$, which is equivalent to the standard norm $\|q\|_{W^{1,p}(L^2)}$ for $q\in \tilde W^{1,p}(L^2)$. Define a mapping $\mathcal M:\tilde W^{1,p}(L^2)\rightarrow \tilde W^{1,p}(L^2)$ by $w=\mathcal M v$ where $w$ satisfies
\begin{equation}\label{Modelw}
\p_tw-\p_t^{1-\alpha_0}  \Delta w=p_v,~~(\bm x,t)\in\Omega\times (0,T],
\end{equation}
equipped with zero initial and boundary conditions. By (\ref{repu}), $w$ could be expressed as
\begin{equation}\label{w}
w=\int_0^tF(t-s)p_v(s)ds.
\end{equation} 
To show the well-posedness of $\mathcal M$, we differentiate (\ref{w}) to obtain
\begin{equation}\label{mh9}
\p_tw=p_v(t)+\int_0^tF'(t-s)p_v(s)ds. 
\end{equation}
 We use the Laplace transform to evaluate the second right-hand side term to get 
\begin{equation}\label{Well:e2}
\ds \hspace{-0.1in} \mathcal L \bigg [ \int_0^{t}  F'(t-s) p_v(s)  d s\bigg ] = z^{\alpha_0}(z^{\alpha_0}-\Delta)^{-1}\mathcal Lp_v
= \big(z^{\alpha_0-\varepsilon}(z^{\alpha_0}-\Delta)^{-1}\big)\big(z^\varepsilon \mathcal Lp_v\big).
\end{equation}
We apply $(I_t^{1-\varepsilon}\p_t^{1-\alpha_0}f)(0)=0$ to take the inverse Laplace transform of \eqref{Well:e2}  to get 
\begin{equation}\label{Well:e3}
\ds \int_0^{t}  F'(t-s) p_v(s)  d s  = \int_0^t \mathcal R(t-s) \partial_s^\varepsilon p_v(s)d s,
\end{equation}
where
$\mathcal R(t):=\frac{1}{2\pi \rm i}\int_{\Gamma_\theta}
z^{\alpha_0-\varepsilon} (z^{\alpha_0}-\Delta)^{-1} e^{zt} \d z.$
We use (\ref{GammaEstimate}) to bound $\mathcal R$ by
\begin{align}
\|\mathcal R(t-s)\| &\le Q\int_{\Gamma_\theta}
|z|^{\alpha_0-\varepsilon} \|(z^{\alpha_0}-\Delta)^{-1}\| |e^{z(t-s)}| |d z|\nonumber\\
&\leq Q\int_{\Gamma_\theta}
|z|^{\alpha_0-\varepsilon} |z|^{-\alpha_0} |e^{z(t-s)}| |d z|\leq  \f{Q}{(t-s)^{1-\varepsilon}}.\label{Well:e4}
\end{align}
To estimate $\partial_t^\varepsilon p_v$, recall that $p_v(t)=\p_t^{1-\alpha_0}f-g'*\p_t v$, which implies
\begin{equation}\label{zxc1}
\partial_t^\varepsilon p_v(t)=\p_t^{\varepsilon}\p_t^{1-\alpha_0}f-\partial_t^\varepsilon(g'*\p_t v). 
\end{equation}
We reformulate  $\partial_t^\varepsilon(g'*\p_t v)$ as
\begin{align}
\partial_t^\varepsilon(g'*\p_t v)=\partial_t\big(\beta_{1-\varepsilon}*(g'*\p_t v)\big)=\partial_t\big((\beta_{1-\varepsilon}*g')*\p_t v\big), \label{mh1}
\end{align}
 where 
\begin{align}\label{mh3}
\beta_{1-\varepsilon}*g'=\int_0^t\frac{(t-s)^{-\varepsilon}}{\Gamma(1-\varepsilon)}g'(s)ds=t^{1-\varepsilon}\int_0^1\frac{(1-z)^{-\varepsilon}}{\Gamma(1-\varepsilon)}g'(tz)dz.
\end{align} 
 By (\ref{g'}), we have
\begin{align}
|\beta_{1-\varepsilon}*g'|&\leq Qt^{1-\varepsilon}\int_0^1\frac{(1-z)^{-\varepsilon}}{\Gamma(1-\varepsilon)}(1+|\ln(tz)|)dz\nonumber\\
&\leq Qt^{1-\varepsilon}\int_0^1(1-z)^{-\varepsilon}\frac{(tz)^{\varepsilon}(1+|\ln(tz)|)}{(tz)^{\varepsilon}}dz \leq Qt^{1-2\varepsilon},\label{mh5}
\end{align}
 which means that $(\beta_{1-\varepsilon}*g')(0)=0$. We apply this to further rewrite (\ref{mh1}) as
 \begin{align}
\partial_t^\varepsilon(g'*\p_t v)=\big(\partial_t(\beta_{1-\varepsilon}*g')\big)*\p_t v. \label{mh2}
\end{align}
 Based on (\ref{mh3}) we apply (\ref{g''}) and a similar estimate as (\ref{mh5}) to obtain
 \begin{align}
|\p_t(\beta_{1-\varepsilon}*g')|&=\Big|(1-\varepsilon)t^{-\varepsilon}\int_0^1\frac{(1-z)^{-\varepsilon}}{\Gamma(1-\varepsilon)}g'(tz)dz\nonumber\\
&\qquad+t^{1-\varepsilon}\int_0^1\frac{(1-z)^{-\varepsilon}}{\Gamma(1-\varepsilon)}g''(tz)zdz\Big|\leq Qt^{-2\varepsilon}.\label{mh4}
\end{align} 
 Thus, we finally bound $\partial_t^\varepsilon(g'*\p_t v)$ as
 \begin{align}\label{mh6}
| \partial_t^\varepsilon(g'*\p_t v)|\leq Qt^{-2\varepsilon}*|\p_t v|.
 \end{align}
 We invoke (\ref{regf}) and (\ref{mh6}) to conclude that
 \begin{align}
 \|e^{-\sigma t}\partial_t^\varepsilon p_v(t)\|_{L^p(L^2)}&\leq \|\p_t^{\varepsilon}\p_t^{1-\alpha_0}f\|_{L^p(L^2)}+Q\|e^{-\sigma t}(t^{-2\varepsilon}*|\p_t v|)\|_{L^p(L^2)}\nonumber\\
 &\leq \|f\|_{W^{1,1}(L^2)}+Q\|(e^{-\sigma t}t^{-2\varepsilon})*(e^{-\sigma t}\|\p_t v\|)\|_{L^p(0,T)}\nonumber\\
 &\leq \|f\|_{W^{1,1}(L^2)}+Q\|e^{-\sigma t}t^{-2\varepsilon}\|_{L^1(0,T)}\|e^{-\sigma t}\|\p_t v\|\|_{L^p(0,T)}\nonumber\\
 &\leq \|f\|_{W^{1,1}(L^2)}+Q\sigma^{2\varepsilon-1}\|v\|_{\tilde W^{1,p}(0,T;L^2)},\label{mh8}
 \end{align}
 where we used the fact that 
 \begin{align}\label{zxc8}
 \|e^{-\sigma t}t^{-2\varepsilon}\|_{L^1(0,T)}=\int_0^Te^{-\sigma t}t^{-2\varepsilon}dt\leq \sigma^{2\varepsilon-1}\int_0^\infty e^{- t}t^{-2\varepsilon}dt\leq Q\sigma^{2\varepsilon-1}. 
 \end{align}
 Consequently, we apply (\ref{Well:e4}), (\ref{mh8}) and the Young's convolution inequality in (\ref{Well:e3}) to bound the second right-hand side term of (\ref{mh9})
 \begin{align*}
& \Big\|e^{-\sigma t}\int_0^{t}  F'(t-s) p_v(s)  d s \Big\|_{L^p(L^2)}\leq Q\|t^{\varepsilon-1}*\|e^{-\sigma t}\partial_t^\varepsilon p_v\|\|_{L^p(0,T)}\\
 &~~~~\leq Q\|e^{-\sigma t}\partial_t^\varepsilon p_v(t)\|_{L^p(L^2)}\leq Q\|f\|_{W^{1,1}(L^2)}+Q\sigma^{2\varepsilon-1}\|v\|_{\tilde W^{1,p}(L^2)}.
 \end{align*}
  The first right-hand side term of (\ref{mh9}) could be bounded by similar and much simper estimates as above
  \begin{align*}
  \|e^{-\sigma t}p_v\|_{L^p(L^2)}&\leq \|\p_t^{1-\alpha_0}f\|_{L^p(L^2)}+\|(e^{-\sigma t}g')*(e^{-\sigma t}\p_t v)\|_{L^p(L^2)} \\
  &\leq  \|\p_t^{1-\alpha_0}f\|_{L^p(L^2)}+Q\sigma^{\varepsilon-1}\|v\|_{\tilde W^{1,p}(L^2)}.
  \end{align*}
  In particular, we could apply $(I_t^{1-\varepsilon}\p_t^{1-\alpha_0}f)(0)=0$ and the Young's convolution inequality to further bound $\p_t^{1-\alpha_0}f$ as
  \begin{align}
   \|\p_t^{1-\alpha_0}f\|_{L^p(L^2)}&= \|\p_tI_t^{\varepsilon}I_t^{1-\varepsilon}\p_t^{1-\alpha_0}f\|_{L^p(L^2)}=\|I_t^{\varepsilon}\p_tI_t^{1-\varepsilon}\p_t^{1-\alpha_0}f\|_{L^p(L^2)}\nonumber\\
   &\hspace{-0.4in}\leq Q\|I_t^{\varepsilon}\|\p_t^{\varepsilon}\p_t^{1-\alpha_0}f\|\|_{L^p(0,T)}\leq Q\|\p_t^{\varepsilon}\p_t^{1-\alpha_0}f\|_{L^p(L^2)}\leq Q\|f\|_{W^{1,1}(L^2)}.\label{fnorm}
   \end{align}
 We summarize the above two estimates to conclude that 
 \begin{equation}\label{mh10}
 \|w\|_{\tilde W^{1,p}(L^2)}\leq Q\big(\|f\|_{W^{1,1}(L^2)}+\sigma^{2\varepsilon-1}\|v\|_{\tilde W^{1,p}(L^2)}\big),
 \end{equation}
 which implies that $w\in \tilde W^{1,p}(L^2)$ such that $\mathcal M$ is well-posed. 
 
 To show the contractivity of $\mathcal M$, let $w_1=\mathcal M v_1$ and $w_2=\mathcal M v_2$ such that $w_1-w_2$ satisfies (\ref{Modelw}) with $v=v_1-v_2$ and $f\equiv 0$. Thus (\ref{mh10}) implies $\|w_1-w_2\|_{\tilde W^{1,p}(L^2)}\leq Q\sigma^{2\varepsilon-1}\|v_1-v_2\|_{\tilde W^{1,p}(L^2)}$. Choose $\sigma$ large enough such that $Q\sigma^{2\varepsilon-1}<1$, that is, $\mathcal M$ is a contraction mapping such that there exists a unique solution $u\in \tilde W^{1,p}(L^2)$ for model (\ref{Model2}) with zero initial and boundary conditions, and the stability estimate could be derived directly from (\ref{mh10}) with $v=w=u$ and large $\sigma$ and the equivalence between two norms $\|\cdot\|_{\tilde W^{1,p}(L^2)}$ and $\|\cdot\|_{W^{1,p}(L^2)}$ for $u\in \tilde W^{1,p}(L^2)$
  \begin{equation}\label{mh11}
 \|u\|_{ W^{1,p}(L^2)}\leq Q\|f\|_{W^{1,1}(L^2)}.
 \end{equation}
 
For the case of non-zero initial condition, a variable substitution $v=u-u_0$ could be used to reach the equation (\ref{Model2}) with $u$ and $f$ replaced by $v$ and $f+\Delta u_0$, respectively, equipped with zero initial and boundary conditions. As $u_0\in H^2$, we apply the well-posedness of (\ref{Model2}) with zero initial and boundary conditions to find that there exists a unique solution $v\in \tilde W^{1,p}(L^2)$ with the stability estimate $ \|v\|_{ W^{1,p}(L^2)}\leq Q\|f+\Delta u_0\|_{W^{1,1}(L^2)}$ derived from (\ref{mh11}).
 By $v=u-u_0$, we finally conclude that $u=v+u_0$ is a solution to (\ref{Model2})--(\ref{ibc}) in $W^{1,p}(L^2)$ with the estimate
 \begin{equation}\label{mh13}
 \|u\|_{ W^{1,p}(L^2)}\leq Q\|f\|_{W^{1,1}(L^2)}+Q\|u_0\|_{H^2}.
 \end{equation}
 The uniqueness of the solutions to (\ref{Model2})--(\ref{ibc}) follows from that to this model with $u_0\equiv 0$. To  estimate the $L^p(H^2)$ norm of $u$, we apply $\Delta$ on both sides of (\ref{repu}) and employ (\ref{propF}) to obtain
 \begin{align}
\|\Delta u\|&\leq \|\Delta F(t)u_0\|+\|\int_0^t\Delta F(t-s)p_u(s)ds\|\nonumber\\
&\leq Q\|F(t)u_0\|_{\check H^2}+Qt^{-\alpha_0}*\|p_u\|\leq Q\|u_0\|_{\check H^2}+Qt^{-\alpha_0}*\|p_u\|.\label{mh14}
\end{align}
We combine this with the Young's convolution inequality, (\ref{fnorm}), (\ref{mh13}) and the estimate of $p_u$
\begin{align}\label{estpu}
\|p_u\|_{L^p(L^2)}&=\|\p_t^{1-\alpha_0}f-g'*\p_t u\|_{L^p(L^2)}\\
&\leq Q\|f\|_{W^{1,1}(L^2)}+Q\|g'\|_{L^1(0,T)}\|\p_tu\|_{L^p(L^2)}\leq Q\|u_0\|_{\check H^2}+Q\|f\|_{W^{1,1}(L^2)}\nonumber
\end{align} 
to get
 \begin{align*}
\|\Delta u\|_{L^p(L^2)}&\leq  Q\|u_0\|_{\check H^2}+Q\|t^{-\alpha_0}*\|p_u\|\|_{L^p(0,T)}\\
&\leq Q\|u_0\|_{\check H^2}+Q\|p_u\|_{L^p(L^2)}\leq  Q\|u_0\|_{\check H^2}+Q\|f\|_{W^{1,1}(L^2)},
\end{align*}
which completes the proof.
\end{proof}

\subsection{Well-posedness of original model}
The well-posedness of the original model (\ref{VtFDEs})--(\ref{ibc}) is proved in the following theorem.
\begin{theorem}\label{thm:well}
Suppose $f\in W^{1,1}(L^2)$ and $u_0\in \check H^2$, then model (\ref{VtFDEs})--(\ref{ibc}) admits a unique solution in $W^{1,p}(L^2)\cap L^p(\check H^2)$ for $1< p<\frac{1}{1-\alpha_0}$ and
\begin{align*}
&\|u\|_{W^{1,p}(L^2)}+\| u\|_{L^p(\check H^2)}\leq Q\big(\|f\|_{W^{1,1}(L^2)}+\|u_0\|_{\check H^2}\big).
\end{align*}
\end{theorem}
\begin{proof}
We first show that a solution to (\ref{Model2})--(\ref{ibc}), which uniquely exists in $W^{1,p}(L^2)\cap L^p(\check H^2)$ by Theorem \ref{thm:aux}, is also a solution to (\ref{VtFDEs})--(\ref{ibc}). If $u\in W^{1,p}(L^2)\cap L^p(\check H^2)$ for $1< p<\frac{1}{1-\alpha_0}$ solves (\ref{Model2})--(\ref{ibc}), then (\ref{Model2}) could be rewritten as
$$\p_t (\beta_{\alpha_0}*k*\p_t u-\beta_{\alpha_0}*  \Delta u  -\beta_{\alpha_0}*f) =0,$$
which means that
\begin{align}\label{mh16}
g*\p_t u-\beta_{\alpha_0}*  \Delta u  -\beta_{\alpha_0}*f=c_0
\end{align}
for some constant $c_0$. We apply the $\|\cdot\|_{L^\infty(0,t;L^2)}$ on both sides and use the boundedness of $g$ to obtain
\begin{align}
|c_0||\Omega|^{1/2}&\leq Q\||g|*\|\p_tu\|\|_{L^\infty(0,t)}+\|\beta_{\alpha_0}*  \Delta u\|_{L^\infty(0,t;L^2)}+Q\|\beta_{\alpha_0}*\|f\|\|_{L^\infty(0,t)}\nonumber\\
&\leq Q\|g\|_{L^\gamma(0,t)}\|\p_tu\|_{L^p(L^2)}+\|\beta_{\alpha_0}*  \Delta u\|_{L^\infty(0,t;L^2)}+Q\|\beta_{\alpha_0}\|_{L^1(0,t)}\|f\|_{L^\infty(L^2)}\nonumber\\
&\leq Qt^{\frac{1}{\gamma}}+Qt^{\alpha_0}+Q\|\beta_{\alpha_0}*  \|\Delta u\|\|_{L^\infty(0,t)}\label{mh15}
\end{align}
for some $\gamma<\infty$ satisfying $\frac{1}{\gamma}+\frac{1}{p}=1$.
To estimate $\|\beta_{\alpha_0}*  \|\Delta u\|\|_{L^\infty(0,t)}$, we employ (\ref{mh14}) to obtain
 \begin{align*}
\beta_{\alpha_0}* \|\Delta u\|&\leq Q\beta_{\alpha_0}* \|u_0\|_{\check H^2}+Q\beta_{\alpha_0}* t^{-\alpha_0}*\|p_u\|\leq Qt^{\alpha_0} \|u_0\|_{\check H^2}+Q(1*\|p_u\|).
\end{align*}
We then take the $\|\cdot\|_{L^\infty(0,t)}$ on both sides and apply (\ref{estpu}) to find
 \begin{align*}
\|\beta_{\alpha_0}* \|\Delta u\|\|_{L^\infty(0,t)}&\leq Qt^{\alpha_0} \|u_0\|_{\check H^2}+Q\|1*\|p_u\|\|_{L^\infty(0,t)}\\
&\hspace{-1in}\leq Qt^{\alpha_0} \|u_0\|_{\check H^2}+Qt^{\frac{1}{\gamma}}\|p_u\|_{L^p(L^2)}\leq  Qt^{\alpha_0} \|u_0\|_{\check H^2}+Qt^{\frac{1}{\gamma}}(\|u_0\|_{\check H^2}+Q\|f\|_{W^{1,1}(L^2)})
\end{align*}
for some $\gamma<\infty$ satisfying $\frac{1}{\gamma}+\frac{1}{p}=1$.
We invoke this in (\ref{mh15}) to finally obtain
$
|c_0||\Omega|^{1/2}\leq Qt^{\frac{1}{\gamma}}+Qt^{\alpha_0}.
$
Passing $t\rightarrow 0+$ implies $c_0=0$, which, together with (\ref{mh16}), gives
$
g*\p_t u-\beta_{\alpha_0}*  \Delta u  -\beta_{\alpha_0}*f=0.
$
We then calculate its convolution with $\beta_{1-\alpha_0}$ and apply $\beta_{1-\alpha_0}*g=\beta_{1-\alpha_0}*\beta_{\alpha_0}*k=1*k$ to obtain
$
1*k*\p_t u-1*  \Delta u  -1*f=0.
$
Differentiating this equation leads to (\ref{VtFDEs}), which implies $u$ solves (\ref{VtFDEs})--(\ref{ibc}). 
 
To prove the uniqueness of the solutions to (\ref{VtFDEs})--(\ref{ibc}) in $W^{1,p}(L^2)\cap L^p(\check H^2)$, suppose (\ref{VtFDEs})--(\ref{ibc}) with $f\equiv 0$ and $u_0\equiv 0$ has a solution $u\in W^{1,p}(L^2)\cap L^p(\check H^2)$. Then $u$ satisfies (\ref{modeln}) with $f\equiv 0$. As $g*\p_t u\in W^{1,p}(L^2)$, so does $\beta_{\alpha_0}*\Delta u$ such that we differentiate (\ref{modeln}) to get (\ref{Model2}) with $f\equiv 0$. Then we apply Theorem \ref{thm:aux} to get $u\equiv 0$, which completes the proof.
\end{proof}

\subsection{Solution regularity}
We give pointwise-in-time estimates of the temporal derivatives of the solutions to model (\ref{VtFDEs})--(\ref{ibc}), which are feasible for numerical analysis. 
\begin{theorem}\label{thm:reg}
Suppose $u_0\in \check H^{2(m+1)}$ and $f\in W^{1,\frac{1}{\alpha_0-\sigma}}(\check H^{2m})$ for $0<\sigma\ll \alpha_0$ for $m=0,1$. Then 
\begin{align*}
\|\p_t u\|_{\check H^{2m}}&\leq   Qt^{\alpha_0-1}(\|u_0\|_{\check H^{2(m+1)}}+\| f\|_{W^{1,\frac{1}{\alpha_0-\sigma}}(\check H^{2m})}),~~a.e. ~t\in (0,T].
\end{align*}
\end{theorem}
\begin{remark}
For $\alpha(t)\equiv\alpha$ for some $0<\alpha<1$, this theorem implies $\p_tu\sim t^{\alpha-1}$, which is consistent with the classical results shown in, e.g. \cite{Luc,SakYam,StyOriGra}, and indicates that the singularity of the solutions to variable-exponent subdiffusion could be essentially characterized by the initial behavior of the variable exponent. 
\end{remark}
\begin{proof}
We mainly prove the case $m=1$ since the case $m=0$ could be proved by the same procedure. As shown in the proof of Theorem \ref{thm:well}, the unique solution to (\ref{Model2})--(\ref{ibc}) in $W^{1,p}(L^2)\cap L^p(\check H^2)$ is also the unique solution to (\ref{VtFDEs})--(\ref{ibc}). Thus we could perform analysis based on (\ref{repu}).  We differentiate both sides of (\ref{repu}) with respect to $t$ and apply $\frac{d}{dt}F(t)=\Delta E(t)$ \cite[Lemma 6.2]{Jinbook} with
$E(t):=\frac{1}{2\pi \rm i}\int_{\Gamma_\theta} e^{tz}(z^{\alpha_0}-\Delta)^{-1}d z $
 to obtain 
$
\p_tu=\Delta E(t)u_0+p_u+\int_0^tF'(t-s)p_u(s)ds.
$
Then we further apply $\Delta$ on both sides of this equation to obtain
\begin{align}\label{mh22}
\p_t\Delta u= E(t)\Delta^2u_0+\tilde p_u+\int_0^tF'(t-s) \tilde p_u(s)ds,
\end{align}
where $\tilde p_u:=\p_t^{1-\alpha_0}\Delta f-g'*\p_t\Delta u$. We apply $\|\cdot\|$ on both sides and invoke $\|E(t)\|\leq Qt^{\alpha_0-1}$ \cite[Theorem 6.4]{Jinbook} to get
 \begin{align}\label{zxc5}
\|\p_t\Delta u\|\leq  Qt^{\alpha_0-1}\|u_0\|_{\check H^4}+\|\tilde p_u\|+\Big\|\int_0^tF'(t-s) \tilde p_u(s)ds\Big\|.
\end{align}
By (\ref{g'}) and
\begin{align}\label{zxc7}
\p_t^{1-\alpha_0}\Delta f=\beta_{\alpha_0}\Delta f(0)+\beta_{\alpha_0}*\p_t\Delta f, 
\end{align}
we have
\begin{align}
\|\tilde p_u\|&\leq \beta_{\alpha_0}\|\Delta f(0)\|+Q\beta_{\alpha_0}*\|\p_t\Delta f\|+Qt^{-\varepsilon}*\|\p_t\Delta u\|\nonumber\\
&\leq \beta_{\alpha_0}\|\Delta f(0)\|+Qt^{\alpha_0-1}\|\Delta f\|_{W^{1,\frac{1}{\alpha_0-\sigma}}(L^2)}+Qt^{-\varepsilon}*\|\p_t\Delta u\|\label{zxc6}
\end{align}
for a.e. $t\in (0,T]$ where we applied the Young's convolution inequality with $\frac{1}{\gamma}+\frac{1}{\frac{1}{\alpha_0-\sigma}}=1$ for some $0<\sigma\ll \alpha_0$ (such that $\gamma<\frac{1}{1-\alpha_0}$) to get
\begin{align}
\|\beta_{\alpha_0}*\|\p_t\Delta f\|\|_{L^\infty(0,t)}&\leq \|\beta_{\alpha_0}\|_{L^\gamma(0,t)}\|\|\p_t\Delta f\|\|_{L^{\frac{1}{\alpha_0-\sigma}}(0,t)}\nonumber\\
&\hspace{-0.5in}\leq Qt^{\alpha_0-1+\frac{1}{\gamma}}\|\Delta f\|_{W^{1,\frac{1}{\alpha_0-\sigma}}(L^2)}\leq Q\|\Delta f\|_{W^{1,\frac{1}{\alpha_0-\sigma}}(L^2)}.\label{zxc4}
\end{align}
We invoke (\ref{zxc6}) in (\ref{zxc5}) and apply the same treatment as (\ref{Well:e2})--(\ref{Well:e4}) to bound the last right-hand side term of (\ref{zxc5}) to obtain for $0<\varepsilon\ll \alpha_0$
\begin{align}\label{mh19}
\|\p_t\Delta u\|&\leq  QMt^{\alpha_0-1}+Qt^{-\varepsilon}*\|\p_t\Delta u\|+Qt^{\varepsilon-1}*\|\p_t^{\varepsilon} \tilde p_u\|.
\end{align}
where $M:=\|u_0\|_{\check H^4}+\|\Delta f\|_{W^{1,\frac{1}{\alpha_0-\sigma}}(L^2)}$.
According to  (\ref{imh}), we have
$\p_t^{\varepsilon}\p_t^{1-\alpha_0}\Delta f=\beta_{\alpha_0-\varepsilon}\Delta f(0)+\beta_{\alpha_0-\varepsilon}*\p_t\Delta f,$
while by the same treatment as (\ref{mh1})--(\ref{mh6}), we have
 $
| \partial_t^\varepsilon(g'*\p_t \Delta u)|\leq Qt^{-2\varepsilon}*|\p_t \Delta u|.
$
We invoke these two equations in (\ref{mh19}) to obtain
\begin{align*}
\|\p_t\Delta u\|&\leq  QMt^{\alpha_0-1}+Qt^{-\varepsilon}*\|\p_t\Delta u\|\\
&\qquad+Q\beta_\varepsilon*[\beta_{\alpha_0-\varepsilon}\|\Delta f(0)\|+\beta_{\alpha_0-\varepsilon}*\|\p_t\Delta f\|+\beta_{1-2\epsilon}*\|\p_t \Delta u\|]\\
&\leq QMt^{\alpha_0-1}+Q[\beta_{\alpha_0}\|\Delta f(0)\|+\beta_{\alpha_0}*\|\p_t\Delta f\|+\beta_{1-\varepsilon}*\|\p_t \Delta u\|],\nonumber
\end{align*}
which, together with (\ref{zxc4}), yields
\begin{align}\label{zxc2}
\|\p_t\Delta u\|&\leq   QMt^{\alpha_0-1}+Q\beta_{1-\varepsilon}*\|\p_t \Delta u\|
\end{align}
for a.e. $t\in (0,T]$, that is,
\begin{align*}
\|e^{-\theta t}t^{1-\alpha_0}\p_t\Delta u\|&\leq   QM+Q\int_0^te^{-\theta (t-s)}\frac{t^{1-\alpha_0}}{s^{1-\alpha_0}}\frac{\|e^{-\theta s}s^{1-\alpha_0}\p_s \Delta u\|}{(t-s)^{\varepsilon}}ds
\end{align*}
for $a.e. ~t\in (0,T]$ and $\theta>0$. For a fixed $\bar t\in (0,T]$, we have for a.e. $t\in (0,\bar t]$
\begin{align*}
\|e^{-\theta t}t^{1-\alpha_0}\p_t\Delta u\|&\leq   QM+Q\int_0^t\frac{t^{1-\alpha_0}}{s^{1-\alpha_0}}\frac{e^{-\theta (t-s)}}{(t-s)^{\varepsilon}}ds\|e^{-\theta t}t^{1-\alpha_0}\p_t \Delta u\|_{L^\infty (0,\bar t;L^2)}
\end{align*}
Note that
\begin{align*}
&\int_0^t\frac{t^{1-\alpha_0}}{s^{1-\alpha_0}}\frac{e^{-\theta (t-s)}}{(t-s)^{\varepsilon}}ds=t^{1-\varepsilon} \int_0^1 e^{-\theta t y}(1-y)^{\alpha_0-1}y^{-\varepsilon}dy\\
&\qquad=\frac{t^{\frac{1-\varepsilon}{2}}}{\theta^{\frac{1-\varepsilon}{2}}} \int_0^1 (t\theta y)^{\frac{1-\varepsilon}{2}} e^{-\theta t y}(1-y)^{\alpha_0-1}y^{-\varepsilon-\frac{1-\varepsilon}{2}}dy\\
&\qquad\leq\frac{t^{\frac{1-\varepsilon}{2}}}{\theta^{\frac{1-\varepsilon}{2}}} \int_0^1  (1-y)^{\alpha_0-1}y^{-\frac{\alpha_0}{2}-\frac{1}{2}}dy\leq \frac{Q}{\theta^{\frac{1-\varepsilon}{2}}}.
\end{align*}
We combine the above two equations to get for a.e. $t\in (0,\bar t]$
\begin{align*}
\|e^{-\theta t}t^{1-\alpha_0}\p_t\Delta u\|&\leq   QM+\frac{Q}{\theta^{\frac{1-\varepsilon}{2}}}\|e^{-\theta t}t^{1-\alpha_0}\p_t \Delta u\|_{L^\infty (0,\bar t;L^2)},
\end{align*}
which implies
\begin{align*}
\|e^{-\theta t}t^{1-\alpha_0}\p_t\Delta u\|_{L^\infty (0,\bar t;L^2)}&\leq   QM+\frac{Q}{\theta^{\frac{1-\varepsilon}{2}}}\|e^{-\theta t}t^{1-\alpha_0}\p_t \Delta u\|_{L^\infty (0,\bar t;L^2)}.
\end{align*}
Selecting $\theta$ large enough we find
$
\|e^{-\theta t}t^{1-\alpha_0}\p_t\Delta u\|_{L^\infty (0,\bar t;L^2)}\leq   QM,
$
which means
\begin{align}\label{zxc3}
\|\p_t\Delta u\|&\leq   QMt^{\alpha_0-1},~~a.e. ~t\in (0,\bar t].
\end{align}
We take $\bar t=T$ to complete the proof.
\end{proof}

\begin{theorem}\label{thm:reg2}
Suppose $u_0\in \check H^{2(m+1)}$ and $f\in W^{2,\frac{1}{\alpha_0-\sigma}}(\check H^{2m})$ for $0<\sigma\ll \alpha_0$ and $m=0,1$. Then 
\begin{align*}
\|\p_t^2 u\|_{\check H^{2m}}&\leq   Qt^{\alpha_0-2}(\|u_0\|_{\check H^{2(m+1)}}+\|f\|_{W^{2,\frac{1}{\alpha_0-\sigma}}(\check H^{2m})}),~~a.e. ~t\in (0,T].
\end{align*}
\end{theorem}
\begin{proof}
We mainly prove the case $m=1$ since the case $m=0$ could be proved by the same procedure.
We apply (\ref{Well:e3}) to (\ref{mh22}) and multiply the resulting equation by $t$ to obtain 
\begin{align*}
t\p_t\Delta u= tE(t)\Delta^2u_0+t\tilde p_u+\int_0^t(t-s) \mathcal R(t-s) \partial_s^\varepsilon \tilde p_u(s)d s+\int_0^t \mathcal R(t-s) s\partial_s^\varepsilon \tilde p_u(s)d s.
\end{align*}
Differentiating this equation with respect to $t$ yields
\begin{align}
&t\p^2_t\Delta u \nonumber\\
&\quad=- \p_t\Delta u+ E(t)\Delta^2u_0+tE'(t)\Delta^2u_0+\p_t(t\tilde p_u)+ \lim_{s\rightarrow t^-}[(t-s)\mathcal R(t-s) \partial_s^\varepsilon\tilde p_u(s)] \nonumber\\
&\qquad+ \int_0^t\mathcal R(t-s) \partial_s^\varepsilon\tilde p_u(s)ds+ \int_0^t(t-s)\mathcal R'(t-s) \partial_s^\varepsilon\tilde p_u(s)ds\nonumber\\
&\qquad+\lim_{s\rightarrow t^-}\mathcal R(s) [\theta\partial_\theta^\varepsilon\tilde p_u]|_{\theta=t-s}+\int_0^t \mathcal R(t-s) \p_s(s\partial_s^\varepsilon\tilde p_u(s))ds=:\sum_{i=1}^9J_i.\label{ttt}
\end{align}
Then we bound the $\|\cdot\|$ norms of $J_1$--$J_9$. By Theorem \ref{thm:reg}, we have 
$\|J_1\|\leq   QMt^{\alpha_0-1}$ for a.e. $t\in (0,T]$.
We then apply $\|E(t)\|+t\|E'(t)\|\leq Qt^{\alpha_0-1}$ \cite[Theorem 6.4]{Jinbook} to bet
$\|J_2+J_3\|\leq Qt^{\alpha_0-1}\|u_0\|_{H^4}. $
To bound $J_4$, we apply (\ref{zxc7}) to obtain
\begin{align*}
t\tilde p_u=t\beta_{\alpha_0}\Delta f(0)+t(\beta_{\alpha_0}*\p_t\Delta f)-(tg')*\p_t\Delta u-g'*(t\p_t\Delta u).
\end{align*}
Differentiate this equation to get
\begin{align*}
\p_t(t\tilde p_u)&=\alpha_0\beta_{\alpha_0}\Delta f(0)+\beta_{\alpha_0}*\p_t\Delta f+t(\beta_{\alpha_0}\p_t\Delta f(0)+\beta_{\alpha_0}*\p_t^2\Delta f)\\
&\qquad-(tg')'*\p_t\Delta u-g'*[\p_t(t\p_t\Delta u)].
\end{align*}
We apply Theorem \ref{thm:reg}, (\ref{g'}), (\ref{g''}) and a similar argument as (\ref{zxc4}) to obtain
\begin{align*}
\|J_4\|&\leq Qt^{\alpha_0-1}\|\Delta f\|_{W^{2,\frac{1}{\alpha_0-\sigma}}(L^2)}+QMt^{-\varepsilon}*t^{\alpha_0-1}+Qt^{-\varepsilon}*\|t\p_t^2\Delta u\|\\
&\leq Q\tilde Mt^{\alpha_0-1}+Qt^{-\varepsilon}*\|t\p_t^2\Delta u\|,
\end{align*}
where $\tilde M:=\|u_0\|_{\check H^4}+\|\Delta f\|_{W^{2,\frac{1}{\alpha_0-\sigma}}(L^2)}$.

To treat the other terms, we apply the same derivations as (\ref{zxc1})--(\ref{mh6}), (\ref{imh}) and Theorem \ref{thm:reg} to obtain for $0<\varepsilon\ll \alpha_0$
\begin{align}
\|\partial_t^\varepsilon \tilde p_u(t)\|&=\|\p_t^{\varepsilon}\p_t^{1-\alpha_0}\Delta f-\partial_t^\varepsilon(g'*\p_t \Delta u)\|\nonumber\\
&=\|\beta_{\alpha_0-\varepsilon}\Delta f(0)+\beta_{\alpha_0-\varepsilon}*\p_t\Delta f-\partial_t^\varepsilon(g'*\p_t \Delta u)\|\nonumber\\
&\leq \beta_{\alpha_0-\varepsilon}\|\Delta f(0)\|+Q\beta_{\alpha_0-\varepsilon}*\|\p_t\Delta f\|+Qt^{-2\varepsilon}*\|\p_t\Delta u\|\nonumber\\
&\leq \beta_{\alpha_0-\varepsilon}\|\Delta f(0)\|+Q\beta_{\alpha_0-\varepsilon}*\|\p_t\Delta f\|+QMt^{\alpha_0-2\varepsilon},\label{mh23}\\
t\partial_t^\varepsilon \tilde p_u(t)&=t\beta_{\alpha_0-\varepsilon}\Delta f(0)+t(\beta_{\alpha_0-\varepsilon}*\p_t\Delta f)\nonumber\\
&\quad-\big(t\partial_t(\beta_{1-\varepsilon}*g')\big)*\p_t \Delta u-\big(\partial_t(\beta_{1-\varepsilon}*g')\big)*(t\p_t \Delta u).\label{mh24}
\end{align}
We then apply (\ref{Well:e4}) and (\ref{mh23}) to bound 
$$\|J_5\|\leq \lim_{s\rightarrow t^-} (t-s)^{\varepsilon} \|\partial_s^\varepsilon \tilde p_u(s)\|=0 $$
and
\begin{align*}
\|J_6\|&\leq Qt^{\varepsilon-1}*\|\partial_t^\varepsilon \tilde p_u(t)\|\leq Q\big[ \beta_{\alpha_0}\|\Delta f(0)\|+\beta_{\alpha_0}*\|\p_t\Delta f\|+Mt^{\alpha_0-\varepsilon}\big]\\
&\leq QM+Q\|\Delta f\|_{W^{2,1+\sigma}(L^2)}.
\end{align*}
To bound $J_7$, we follow the idea proposed in, e.g. \cite[Page 190]{Jinbook}, to set $\delta=t^{-1}$ in $\Gamma_\theta$ and apply the variable substitution $z=s\cos\phi+\rm is\sin\phi$ to bound $\mathcal R'(t)$ as
\begin{align*}
\|\mathcal R'(t)\|=\Big\|\frac{1}{2\pi \rm i}\int_{\Gamma_\theta}
z^{1+\alpha_0-\varepsilon} (z^{\alpha_0}-\Delta)^{-1} e^{zt} \d z\Big\|\leq \int_{\Gamma_\theta}e^{\mathfrak{R}(z)t}|z|^{1-\varepsilon}|dz|\leq Qt^{\varepsilon-2}.
\end{align*}
We apply this and a similar estimate as $J_6$ to bound $J_7$ as
$$
\|J_7\|\leq Qt^{\varepsilon-1}*\|\partial_t^\varepsilon\tilde p_u(t)\|\leq QM+Q\|\Delta f\|_{W^{2,1+\sigma}(L^2)}.
$$
To bound $J_8$, we apply (\ref{mh2}) to bound (\ref{mh24}) as
\begin{align*}
\|t\partial_t^\varepsilon \tilde p_u(t)\|&\leq t\beta_{\alpha_0-\varepsilon}\|\Delta f(0)\|+Qt(\beta_{\alpha_0-\varepsilon}*\|\p_t\Delta f\|)\\
&\quad+Qt^{1-2\varepsilon}*\|\p_t \Delta u\|+Qt^{-2\varepsilon}*\|t\p_t \Delta u\|\leq Q(M+\|\Delta f\|_{W^{2,1+\sigma}(L^2)})t^{\alpha_0-\varepsilon},
\end{align*}
which, together with (\ref{Well:e4}), implies
$$
\|J_8\|\leq \lim_{s\rightarrow t^-}Qs^{\varepsilon-1} (t-s)^{\alpha_0-\varepsilon}=0.
$$
To bound $J_9$, we differentiate (\ref{mh24}) with respect to $t$ to obtain
\begin{align}
\p_t(t\partial_t^\varepsilon \tilde p_u)&=\p_t(t\beta_{\alpha_0-\varepsilon})\Delta f(0)+\beta_{\alpha_0-\varepsilon}*\p_t\Delta f+t\p_t(\beta_{\alpha_0-\varepsilon}*\p_t\Delta f)\nonumber\\
&\quad-\big[\p_t\big(t\partial_t(\beta_{1-\varepsilon}*g')\big)\big]*\p_t \Delta u-\big(\partial_t(\beta_{1-\varepsilon}*g')\big)*[\p_t(t\p_t \Delta u)].\label{mh26}
\end{align}
By (\ref{mh4}), we have
 \begin{align*}
t\p_t(\beta_{1-\varepsilon}*g')&=(1-\varepsilon)t^{1-\varepsilon}\int_0^1\frac{(1-z)^{-\varepsilon}}{\Gamma(1-\varepsilon)}g'(tz)dz+t^{2-\varepsilon}\int_0^1\frac{(1-z)^{-\varepsilon}}{\Gamma(1-\varepsilon)}g''(tz)zdz,
\end{align*} 
which, together with (\ref{g''}) and (\ref{g'''}), implies
 \begin{align*}
&|\p_t(t\p_t(\beta_{1-\varepsilon}*g'))|\\
&\quad=\Big|(1-\varepsilon)^2t^{-\varepsilon}\int_0^1\frac{(1-z)^{-\varepsilon}}{\Gamma(1-\varepsilon)}g'(tz)dz+(1-\varepsilon)t^{1-\varepsilon}\int_0^1\frac{(1-z)^{-\varepsilon}}{\Gamma(1-\varepsilon)}g''(tz)zdz\\
&\qquad+(2-\varepsilon)t^{1-\varepsilon}\int_0^1\frac{(1-z)^{-\varepsilon}}{\Gamma(1-\varepsilon)}g''(tz)zdz+t^{2-\varepsilon}\int_0^1\frac{(1-z)^{-\varepsilon}}{\Gamma(1-\varepsilon)}g'''(tz)z^2dz\Big|\leq Qt^{-2\varepsilon}.
\end{align*} 
Furthermore, 
$
|t\p_t(\beta_{\alpha_0-\varepsilon}*\p_t\Delta f)|\leq Qt^{\alpha_0-\varepsilon}|(\p_t\Delta f)(0)|+Qt(\beta_{\alpha_0-\varepsilon}*|\p_t^2\Delta f|).
$
We incorporate these estimates and (\ref{mh4}) with (\ref{mh26}) to obtain
\begin{align*}
\|\p_t(t\partial_t^\varepsilon \tilde p_u)\|&\leq Qt^{\alpha_0-\varepsilon-1}\|\Delta f(0)\|+Q\beta_{\alpha_0-\varepsilon}*\|\p_t\Delta f\|\\
&\hspace{-0.65in}+Qt^{\alpha_0-\varepsilon}\|(\p_t\Delta f)(0)\|+Qt(\beta_{\alpha_0-\varepsilon}*\|\p_t^2\Delta f\|)+Qt^{-2\varepsilon}*\|\p_t \Delta u\|+Qt^{-2\varepsilon}*\|t\p_t^2 \Delta u\|\\
&\leq Q(1+t^{\alpha_0-\varepsilon-1})\|\Delta f\|_{W^{2,\frac{1}{\alpha_0-\sigma}}(L^2)}+QM+Qt^{-2\varepsilon}*\|t\p_t^2 \Delta u\|  
\end{align*}
for a.e. $t\in (0,T]$ where we used the Young's convolution inequality to bound  $\beta_{\alpha_0-\varepsilon}*\|\p_t^2\Delta f\|$ as (\ref{zxc4}).
We invoke this and (\ref{Well:e4}) to bound $J_9$ as
\begin{align*}
\|J_9\| &\leq Qt^{\alpha_0-1}\|\Delta f\|_{W^{2,\frac{1}{\alpha_0-\sigma}}(L^2)}+QM+Qt^{-\varepsilon}*\|t\p_t^2 \Delta u\| .
\end{align*}
We invoke the estimates of $J_1$--$J_9$ in (\ref{ttt}) to conclude that
\begin{align*}
\|t\p^2_t\Delta u\|\leq Q\tilde Mt^{\alpha_0-1}+Qt^{-\varepsilon}*\|t\p_t^2 \Delta u\|, ~a.e.~t\in (0,T].
\end{align*}
Then we follow exactly the same procedure as (\ref{zxc2})--(\ref{zxc3}) to complete the proof of this theorem.
\end{proof}
\subsection{An inverse problem}
From previous results, we notice that the initial value $\alpha_0$ plays a critical role in determining the properties of the model. Thus, it is meaningful to determine $\alpha_0$ from observations of $u$, which formulates an inverse problem. Here we follow the proof of \cite[Theorem 6.31]{Jinbook} to present a preliminary result for demonstrating the usage of equivalent formulations derived by the generalized identity function.  
\begin{theorem}
For (\ref{VtFDEs})--(\ref{ibc}) with $f\equiv 0$, $u_0\in C^\infty_0$ and $\Delta u_0(x_0)\neq 0$ for some $x_0\in\Omega$, we have
$$\alpha_0=\lim_{t\rightarrow 0^+} \frac{t\p_tu(x_0,t)}{u(x_0,t)-u_0(x_0)}. $$  
\end{theorem}
\begin{proof}
By Theorem \ref{thm:well} the model (\ref{VtFDEs})--(\ref{ibc}) is well-posed and could be reformulated as (\ref{Model2}). We then further calculate the convolution of (\ref{Model2}) with $f\equiv 0$ and $\beta_{1-\alpha_0}$ to get
\begin{equation}\label{ip1}
\beta_{1-\alpha_0}*(\p_tu)+\beta_{1-\alpha_0}*(g'*\p_t u)-  \Delta u=0,
\end{equation}
the solution of which could be expressed as
$u=F(t)u_0+\mathcal H $ where $\mathcal H$ could be expressed via the Mittag-Leffler function $E_{a,b}(\cdot)$ \cite{Gor}
\begin{align*}
\mathcal H:=\sum_{j=1}^\infty \rho*(-\beta_{1-\alpha_0}*(g'*\p_t u),\phi_j)\phi_j,~~\rho:=t^{\alpha_0-1}E_{\alpha_0,\alpha_0}(-\lambda_j t^{\alpha_0}).
\end{align*}
Also, (\ref{ip1}) implies $\Delta u|_{\p\Omega}=0$. Then one could follow the proof of Theorem \ref{thm:reg} to further show that $\|\p_t\Delta^2u\|\leq Qt^{\alpha_0-1}$, which will be used later.

By the derivations of \cite[Theorem 6.31]{Jinbook} we find that if we can show 
\begin{equation}\label{ip2}
\lim_{t\rightarrow 0^+}t^{1-\alpha_0}\p_t\mathcal H(x_0,t)=0\text{ and }\lim_{t\rightarrow 0^+}t^{-\alpha_0}\mathcal H(x_0,t)=0,
\end{equation}
 then the proof of  \cite[Theorem 6.31]{Jinbook} is not affected at all by the term $\mathcal H$ such that we immediately reach the conclusion by \cite[Theorem 6.31]{Jinbook}. 

To show the first statement, we first apply integration by parts to obtain 
\begin{align*}
\mathcal H=\sum_{j=1}^\infty \frac{1}{\lambda_j^2}\rho*(-\beta_{1-\alpha_0}*(g'*\p_t \Delta^2u),\phi_j)\phi_j,
\end{align*}
and we follow similar derivations as (\ref{mh3})--(\ref{mh4}) to obtain $|\p_t(\beta_{1-\alpha_0}*g')|\leq t^{-\alpha_0-\varepsilon}$ for $0<\epsilon\ll 1-\alpha_0$. We invoke this and apply $\|\phi_j\|_{L^\infty}\leq Q\lambda_j^{\frac{\kappa}{2}}$ for $\kappa>\frac{d}{2}$ \cite[Page 257]{Jinbook} to bound $\p_t\mathcal H$ as
\begin{align*}
|\p_t\mathcal H|&\leq Q\sum_{j=1}^\infty \frac{1}{\lambda_j^{2-\frac{\kappa}{2}}}\rho*\big(|\p_t(\beta_{1-\alpha_0}*g')|*\|\p_t \Delta^2u\|\big)\\
&\leq Q\sum_{j=1}^\infty \frac{1}{\lambda_j^{2-\frac{\kappa}{2}}}\rho*\big(t^{-\alpha_0-\varepsilon}*t^{\alpha_0-1}\big)\leq Q\sum_{j=1}^\infty \frac{1}{\lambda_j^{2-\frac{\kappa}{2}}}\rho*t^{-\varepsilon}.
\end{align*}
We combine this with the following estimates \cite{Gor}
$$\rho*t^{-\varepsilon}=\Gamma(\alpha_0)t^{\alpha_0-\varepsilon} E_{\alpha_0,1+\alpha_0-\varepsilon}(-\lambda_jt^{\alpha_0})\leq Q\frac{t^{\alpha_0-\varepsilon}}{1+\lambda_jt^{\alpha_0}}=Q\frac{t^{-\varepsilon}}{\lambda_j}\frac{\lambda_j t^{\alpha_0}}{1+\lambda_jt^{\alpha_0}}\leq Q\frac{t^{-\varepsilon}}{\lambda_j} $$
to obtain
\begin{align*}
t^{1-\alpha_0}|\p_t\mathcal H|&\leq Qt^{1-\alpha_0-\varepsilon}\sum_{j=1}^\infty \frac{1}{\lambda_j^{3-\frac{\kappa}{2}}}\leq Qt^{1-\alpha_0-\varepsilon}\sum_{j=1}^\infty \frac{1}{j^{\frac{2}{d}(3-\frac{\kappa}{2})}}
\end{align*}
where we used the Weyl's law $\lambda_j\sim j^{\frac{2}{d}}$ for $j\rightarrow \infty$. We set $\kappa=\frac{d}{2}+1$ to obtain $\frac{2}{d}(3-\frac{\kappa}{2})=\frac{5}{d}-\frac{1}{2}>1$ for $1\leq d\leq 3$. Thus we have $t^{1-\alpha_0}|\p_t\mathcal H|\leq Qt^{1-\alpha_0-\varepsilon}$, which implies the first statement in (\ref{ip2}). The second statement could be proved similarly and we thus omit the details.
\end{proof}
\section{Numerical approximation}
\subsection{A numerically-feasible formulation}
We turn the attention to numerical approximation. As the kernel $k(t)$ in (\ref{VtFDEs}) may not be positive definite or monotonic, it is not convenient to perform numerical analysis for numerical methods of model (\ref{VtFDEs}). Here a convolution kernel $\beta(t)$ is said to be positive definite if for each $0<\bar t\leq T$ 
$$\int_0^{\bar t}q(t)(\beta*q)(t)dt\geq 0,~~\forall q\in C[0,\bar t], $$
and $\beta_{\mu,\lambda}(t):=e^{-\lambda t}\beta_\mu(t)$ with $0<\mu<1$ and $\lambda\geq 0$ is a typical positive definite convolution kernel, see e.g. \cite{Mus}. 

To find a feasible model for numerical discretization, we recall that the convolution of (\ref{VtFDEs}) and $\beta_{\alpha_0}$ generates (\ref{modeln}), and we apply the integration by parts 
$g*\p_tu=u(t)-g(t)u_0+g'*u$
to rewrite (\ref{modeln}) as 
\begin{align}\label{Modeln}
u+g'*u-\beta_{\alpha_0}*  \Delta u  =gu_0+\beta_{\alpha_0}*f. 
\end{align}
It is clear that the solution $u\in W^{1,p}(L^2)\cap L^p(\check H^2)$ to model (\ref{VtFDEs})--(\ref{ibc}) solves (\ref{Modeln})--(\ref{ibc}). Then we intend to show that (\ref{Modeln})--(\ref{ibc}) has a unique solution in $ W^{1,p}(L^2)\cap L^p(\check H^2)$ such that both (\ref{VtFDEs})--(\ref{ibc}) and (\ref{Modeln})--(\ref{ibc}) have the same unique solution in $ W^{1,p}(L^2)\cap L^p(\check H^2)$. 

Suppose $u_1,u_2\in W^{1,p}(L^2)\cap L^p(\check H^2)$ solves (\ref{Modeln})--(\ref{ibc}), then $\varepsilon_u=u_1-u_2$ solves
$$
\varepsilon_u+g'*\varepsilon_u-\beta_{\alpha_0}*  \Delta \varepsilon_u =0
$$
with homogeneous initial and boundary conditions. As the first two terms belong to $W^{1,p}(L^2)$, so does the third term such that we differentiate this equation to get
$$
\p_t\varepsilon_u+g'*\p_t\varepsilon_u-\p_t^{1-\alpha_0}*\Delta \varepsilon_u =0,
$$
that is, $\varepsilon_u$ satisfies the transformed model (\ref{Model2}) with $f\equiv 0$ and homogeneous initial and boundary conditions. Then Theorem \ref{thm:aux} leads to $\|\varepsilon_u\|_{W^{1,p}(L^2)}=\|\varepsilon_u\|_{L^p(\check H^2)}=0$, which implies that it suffices to numerically solve (\ref{Modeln})--(\ref{ibc}) in order to get the solutions to the original model (\ref{VtFDEs})--(\ref{ibc}). Furthermore, one could apply the substitution $\tilde u= u-u_0$ to transfer the initial condition of (\ref{Modeln})--(\ref{ibc}) to $0$ in order to apply the positive-definite quadrature rules
\begin{align}\label{ModelN}
&\tilde u+g'*\tilde u-\beta_{\alpha_0}*  \Delta \tilde u  =\mathcal F:=\beta_{\alpha_0}*(\Delta u_0+f),~~(\bm x,t)\in \Omega\times (0,T];\\
 &\qquad\qquad\tilde u(0)=0,~\bm x\in \Omega; \quad \tilde u = 0,~(\bm x,t) \in \p \Omega\times[0,T].\label{ibc0}
\end{align}
For the sake of numerical analysis, we further multiply $e^{-\lambda t}$ on both sides of (\ref{ModelN}) to get an equivalent model with respect to $v=e^{-\lambda t}\tilde u$
\begin{align}\label{ModelNE}
&v+(e^{-\lambda t}g')*v-\beta_{\alpha_0,\lambda}*  \Delta v  =e^{-\lambda t}\mathcal F,~~(\bm x,t)\in \Omega\times (0,T];\\
 &\qquad ~v(0)=0,~\bm x\in \Omega; \quad v = 0,~(\bm x,t) \in \p \Omega\times[0,T].\label{ibc00}
\end{align}
Note that $v=e^{-\lambda t}(u-u_0)$ has the same regularity as $u$ and if we find the numerical solution to (\ref{ModelNE})--(\ref{ibc00}), we could immediately define the numerical solution to (\ref{ModelN})--(\ref{ibc0}) by the relation $u=e^{\lambda t}v+u_0$ as shown later.
\subsection{Semi-discrete scheme and error estimate}
Define a quasi-uniform partition of $\Omega$ with mesh diameter $h$ and let $S_h$ be the space of continuous and piecewise linear functions on $\Omega$ with respect to the partition. Let $I$ be the identity operator. The Ritz projection $\Pi_h:H^1_0\rightarrow S_h$ defined by
$\big(\nabla(q-\Pi_h q),\nabla \chi\big)=0$ for any $ \chi\in S_h$ has the approximation property
$
\big \| (I - \Pi_h) q \big \| \leq Q h^2 \|q\|_{H^2}$ for $q\in H^2\cap H_0^1$ \cite{Tho}.

The weak formulation of (\ref{ModelNE})--(\ref{ibc00}) reads for any $\chi\in H^1_0$
\begin{align}\label{wf0}
(v,\chi)+\big((e^{-\lambda t}g')*v,\chi)+(\beta_{\alpha_0,\lambda}*  \nabla v,\nabla \chi)  =(e^{-\lambda t}\mathcal F,\chi),
\end{align}
and the corresponding semi-discrete finite element scheme is: find $v_h\in S_h$ such that
\begin{align}\label{sd0}
(v_h,\chi)+\big((e^{-\lambda t}g')*v_h,\chi)+(\beta_{\alpha_0,\lambda}*  \nabla v_h,\nabla \chi)  =(e^{-\lambda t}\mathcal F,\chi),~~\forall \chi\in S_h.
\end{align}
After obtaining $v_h$, define the approximation $u_h$ to the solution of (\ref{ModelN})--(\ref{ibc0}) as 
$$u_h=e^{\lambda t}v_h+\Pi_hu_0.$$

\begin{theorem}
Suppose $u_0\in \check H^{4}$ and $f\in W^{1,\frac{1}{\alpha_0-\sigma}}(\check H^{2})$ for $0<\sigma\ll \alpha_0$. Then the following error estimate holds
 \begin{align*}
&\|u-u_h\|_{L^2(L^2)}\leq Qh^2.
\end{align*}
\end{theorem}
\begin{proof}
Set $v-v_h=\xi+\eta$ with $\eta=v-\Pi_h v$ and $\xi=\Pi_h v-v_h$. Then we subtract (\ref{wf0}) from (\ref{sd0}) and select $\chi=\xi$ to obtain
\begin{align*}
(\xi,\xi)+\big((e^{-\lambda t}g')*\xi,\xi)+(\beta_{\alpha_0,\lambda}*  \nabla \xi,\nabla \xi)  =-(\eta,\xi)-\big((e^{-\lambda t}g')*\eta,\xi),
\end{align*}
which, together with $|ab|\leq a^2+\frac{b^2}{4}$, implies
\begin{align*}
\|\xi\|^2+4(\beta_{\alpha_0,\lambda}*  \nabla \xi,\nabla \xi)  \leq 4\|(e^{-\lambda t}g')*\xi\|^2 +4\|\eta\|^2+4\|(e^{-\lambda t}g')*\eta\|^2.
\end{align*}
Integrate this equation on $ (0,T)$ and apply 
the positive definiteness of $\beta_{\alpha_0,\lambda}$
 to obtain
\begin{align}\label{sd2}
\|\xi\|^2_{L^2(L^2)}  \leq 4\|(e^{-\lambda t}g')*\xi\|^2_{L^2(L^2)} +4\|\eta\|^2_{L^2(L^2)}+4\|(e^{-\lambda t}g')*\eta\|^2_{L^2(L^2)}.
\end{align}
By Theorem \ref{thm:reg} we have
$$\|u\|_{\check H^2}=\|\Delta u\|\leq \|\Delta u_0\|+Q\int_0^t\|\p_s\Delta u(s)\|ds\leq \|\Delta u_0\|+Q\int_0^ts^{\alpha_0-1}ds\leq Q. $$
Thus, $\|v\|_{\check H^2}\leq Q$ such that 
\begin{align}\label{etaest}
\|\eta\|\leq Qh^2,
\end{align}
 which, together with (\ref{sd2}), leads to
\begin{align*}
\|\xi\|^2_{L^2(L^2)}  \leq 4\|(e^{-\lambda t}g')*\xi\|^2_{L^2(L^2)} +Qh^4.
\end{align*}
We bound the first right-hand side term as
\begin{align*}
\|(e^{-\lambda t}g')*\xi\|^2_{L^2(L^2)} &\leq  Q\big(\| |e^{-\lambda t}g'|*\|\xi\|\|_{L^2(0,T)}\big)^2\leq Q\|e^{-\lambda t}g'\|^2_{L^1(0,T)}\|\xi\|^2_{L^2(L^2)}.
\end{align*}
By (\ref{g'}) and a similar argument as (\ref{zxc8}), we have $$\|e^{-\lambda t}g'\|_{L^1(0,T)}\leq Q\|e^{-\lambda t}t^{-\varepsilon}\|_{L^1(0,T)}\leq Q\lambda^{\varepsilon-1}$$
 for $0<\varepsilon\ll 1$. We combine the above three estimates and set $\lambda$ large enough to get $\|\xi\|_{L^2(L^2)}  \leq Qh^2$ and thus $\|v-v_h\|_{L^2(L^2)}  \leq Qh^2$. We finally apply $u=e^{\lambda t}v+u_0$ and $u_h=e^{\lambda t}v_h+\Pi_h u_0$ to complete the proof.
\end{proof}

\subsection{Fully-discrete scheme and error estimate}
Partition $[0,T]$ by $t_n := n\tau$ for $\tau := T/N$ and $0\leq n\leq N$. Let $I_\tau$ be the piecewise linear interpolation operator with respect to this partition. 
We approximate the convolution terms in (\ref{ModelNE}) by the quadrature rule proposed in, e.g. \cite[Equation (4.16)]{MclTho} 
\begin{align*}
((e^{-\lambda t}g')*v)(t_n)&=((e^{-\lambda t}g')*I_\tau v)(t_n)+R_n=C_n(v)+R_n;\\
(-\beta_{\alpha_0,\lambda}*  \Delta v)(t_n) &=(\beta_{\alpha_0,\lambda}*I_\tau(-  \Delta v))(t_n)+\tilde R_n=\tilde C_n(-\Delta v)+\tilde R_n,
\end{align*}
where 
\begin{align*}
C_n(v)&:=\sum_{i=1}^n\kappa_{n,i}v(t_i),~~\kappa_{n,i}:=\frac{1}{\tau}\int_{t_{n-1}}^{t_n}\int_{t_{i-1}}^{\min\{t_i,t\}}e^{-\lambda (t-s)}g'(t-s)dsdt;\\
\tilde C_n(-\Delta v)&:=\sum_{i=1}^n\tilde \kappa_{n,i}(-\Delta v(t_i)),~~\tilde \kappa_{n,i}:=\frac{1}{\tau}\int_{t_{n-1}}^{t_n}\int_{t_{i-1}}^{\min\{t_i,t\}}\beta_{\alpha_0,\lambda}(t-s)dsdt
\end{align*}
with truncation errors $R_n$ and $\tilde R_n$, respectively. 
Note that $\kappa_{n,i}$ is indeed a function of $n-i$. To see this, direct calculations yield
\begin{align*}
\kappa_{n,i}&=\frac{1}{\tau}\int_{t_{n-1}}^{t_n}\int^{t-t_{i-1}}_{t-\min\{t_i,t\}}e^{-\lambda y}g'(y)dydt=\frac{1}{\tau}\int_{-\tau}^{0}\int^{t_n+z-t_{i-1}}_{t_n+z-\min\{t_i,t_n+z\}}e^{-\lambda y}g'(y)dydz.
\end{align*}
It is clear that $t_n-t_{i-1}=(n-i+1)\tau$, and $t_n+z-\min\{t_i,t_n+z\}=(n-i)\tau+z$ for $i\leq n-1$ and  $t_n+z-\min\{t_i,t_n+z\}=0$ for $i=n$. Thus we have $\kappa_{n,i}=\bar\kappa_{n-i}$ where $\bar\kappa_m$ for $0\leq m\leq N-1$ is defined as
\begin{align}\label{kap1}
\bar\kappa_m=\frac{1}{\tau}\int_{-\tau}^{0}\int^{(m+1)\tau+z}_{m\tau+z}e^{-\lambda y}g'(y)dydz,~~1\leq m\leq N-1
\end{align}
and
\begin{align}\label{kap2}
\bar\kappa_0=\frac{1}{\tau}\int_{-\tau}^{0}\int^{\tau+z}_{0}e^{-\lambda y}g'(y)dydz.
\end{align}
It is shown in \cite[Lemma 4.8]{MclTho} that the quadrature rule  $\tilde C_n$ is weakly positive such that, since $v(0)=0$, the following relation holds
\begin{align}\label{pd}
\sum_{n=1}^{n^*}\tilde C_n(\Phi)\Phi_n\geq 0,~~\text{ for }1\leq n^*\leq N
\end{align}
where $\Phi=\{\Phi_1,\cdots,\Phi_N\}$. Furthermore, \cite[Lemma 4.8]{MclTho} gives the estimates of the truncation errors as follows
\begin{align}
|R_n|&\leq 2\mu_{n-1}\int_0^\tau|\partial_t v|dt+\tau\sum_{j=2}^n\mu_{n-j}\int_{t_{j-1}}^{t_j}|\partial _t^2v|dt;\nonumber\\
|\tilde R_n|&\leq 2\tilde \mu_{n-1}\int_0^\tau|\partial_t\Delta v|dt+\tau\sum_{j=2}^n\tilde \mu_{n-j}\int_{t_{j-1}}^{t_j}|\partial _t^2\Delta v|dt;\label{tRn}\\
\mu_j&:=\int_{t_j}^{t_{j+1}}|e^{-\lambda s}g'(s)|ds,~~\tilde \mu_j:=\int_{t_j}^{t_{j+1}}|\beta_{\alpha_0,\lambda}(s)|ds.\nonumber
\end{align}
Thus the weak formulation of (\ref{ModelNE})--(\ref{ibc00}) reads for any $\chi\in H^1_0$
 \begin{align}\label{wf}
(v_n,\chi)+(C_n(v),\chi) +(\tilde C_n(\nabla v),\nabla \chi) =(e^{-\lambda t_n}\mathcal F_n,\chi)-(R_n+\tilde R_n,\chi),
\end{align}
and the corresponding fully-discrete finite element method reads: find $V_n\in S_h$ for $1\leq n\leq N$ such that
 \begin{align}\label{fd}
(V_n,\chi)+(C_n(V),\chi) +(\tilde C_n(\nabla V),\nabla \chi) =(e^{-\lambda t_n}\mathcal F_n,\chi),~~\forall \chi\in S_h.
\end{align}
After solving this scheme for $V$, we define the numerical approximation for the solution $u$ of (\ref{ModelN})--(\ref{ibc0}) as 
\begin{align}\label{U}
U_n:=e^{\lambda t_n}V_n+\Pi_h u_0,~~1\leq n\leq N. 
\end{align}

Define the discrete-in-time $l^2$ norm as $\|\Phi\|^2_{l^2(0,t_n;L^2)}:=\tau\sum_{i=1}^n\|\Phi_i\|^2$. We prove error estimate of $u-U$ under the norm $\|\cdot\|_{l^2(0,T;L^2)}$ in the following theorem.
\begin{theorem}\label{thm:err}
Suppose $u_0\in \check H^{4}$ and $f\in W^{2,\frac{1}{\alpha_0-\sigma}}(\check H^{2})$ for $0<\sigma\ll \alpha_0$. Then the following error estimate holds
 \begin{align*}
&\|u-U\|_{l^2(0,T;L^2)}\leq Qh^2+Q\tau^{\frac{1}{2}+\frac{3}{2}\alpha_0}.
\end{align*}
\end{theorem}
\begin{proof}
Denote $v_n-V_n=\xi_n+\eta_n$ with $\eta_n:=v_n-\Pi_h v_n$ and $\xi_n:=\Pi_h v_n-V_n$. We subtract (\ref{wf}) from (\ref{fd}) with $\chi=\xi_n$  to obtain
 \begin{align*}
&(\xi_n,\xi_n) +(\tilde C_n(\nabla \xi),\nabla \xi_n)\\
 &\qquad=-(\eta_n,\xi_n)-(C_n(\xi),\xi_n)-(C_n(\eta),\xi_n)-(R_n+\tilde R_n,\xi_n).
\end{align*}
We sum this equation multiplied by $\tau$ from $n=1$ to $n^*$ for some $n^*\leq N$ and apply (\ref{pd}) and $|ab|\leq \frac{1}{8}a^2+2b^2$ to obtain
 \begin{align}
&\|\xi\|^2_{l^2(0,t_{n^*};L^2)}\leq 2\|\eta_n\|^2_{l^2(0,t_{n^*};L^2)}+2\|C_n(\xi)\|^2_{l^2(0,t_{n^*};L^2)}\nonumber\\
 &\qquad+2\|C_n(\eta)\|^2_{l^2(0,t_{n^*};L^2)}+\frac{1}{2}\|\xi\|^2_{l^2(0,t_{n^*};L^2)}+2\|R_n+\tilde R_n\|^2_{l^2(0,t_{n^*};L^2)}.\label{err2}
\end{align}
Then we intend to bound the right-hand side terms. 
By (\ref{etaest}) we have
\begin{align}\label{usc}
\|C_n(\eta)\|\leq\sum_{i=1}^n|\kappa_{n,i}|\|\eta_i\|\leq Qh^2\frac{1}{\tau}\int_{t_{n-1}}^{t_n}\int_{0}^{t}e^{-\lambda (t-s)}|g'(t-s)|dsdt\leq Qh^2,
\end{align}
which leads to
$
\|\eta_n\|^2_{l^2(0,t_{n^*};L^2)}+\|C_n(\eta)\|^2_{l^2(0,t_{n^*};L^2)}\leq Qh^4.
$

By Theorems \ref{thm:reg} and \ref{thm:reg2} and $v=e^{-\lambda t}u$, we have $\|\Delta^m \p_tv\|\leq Qt^{\alpha_0-1}$ and $\|\Delta^m \p^2_tv\|\leq Qt^{\alpha_0-2}$ for $m=0,1$. Also, the fact that $|t^{\alpha_0}-\bar t^{\alpha_0}|\leq |t-\bar t|^{\alpha_0}$ for $t,\bar t\in [0,T]$ implies $\tilde \mu_n\leq Q\tau^{\alpha_0}$. We apply them  to bound $\|\tilde R_n\|$ by (\ref{tRn})
\begin{align}
\|\tilde R_n\|&\leq Q\tilde \mu_{n-1}\int_0^\tau t^{\alpha_0-1}dt+Q\tau\sum_{j=2}^n\tilde \mu_{n-j}\int_{t_{j-1}}^{t_j}t^{\alpha_0-2}dt\label{er3}\\
&\leq Q\tau^{2\alpha_0}+Q\tau^{1+\alpha_0}\int_{t_{1}}^{t_n}t^{\alpha_0-2}dt\leq Q\tau^{2\alpha_0}+Q\tau^{1+\alpha_0}t_1^{\alpha_0-1}\leq Q\tau^{2\alpha_0}.\label{er4}
\end{align}
Then we invoke both (\ref{er3}) and (\ref{er4}) to bound $\|\tilde R_n\|^2_{l^2(0,t_{n^*};L^2)}$ as
\begin{align*}
\|\tilde R_n\|^2_{l^2(0,t_{n^*};L^2)}&\leq \tau\sum_{n=1}^{n_*}\|\tilde R_n\|^2\leq Q\tau^{1+2\alpha_0}\sum_{n=1}^{n_*}\|\tilde R_n\|\\
&\hspace{-0.5in}\leq Q\tau^{1+2\alpha_0}\Big[\sum_{n=1}^{n_*}\tilde \mu_{n-1}\int_0^\tau t^{\alpha_0-1}dt+\tau\sum_{n=2}^{n_*}\sum_{j=2}^n\tilde \mu_{n-j}\int_{t_{j-1}}^{t_j}t^{\alpha_0-2}dt\Big]\\
&\hspace{-0.5in}\leq Q\tau^{1+2\alpha_0}\Big[\tau^{\alpha_0}+\tau\int_{t_1}^{t_{n_*}}t^{\alpha_0-2}dt\Big]\leq Q\tau^{1+3\alpha_0}.
\end{align*}
The $\|R_n\|^2_{l^2(0,t_{n^*};L^2)}$ could be estimated following exactly the same procedure with the bound dominated by $Q\tau^{1+3\alpha_0}$.
We invoke the above estimates in (\ref{err2}) to get
 \begin{align*}
&\|\xi\|^2_{l^2(0,t_{n^*};L^2)}\leq Qh^4+Q\tau^{1+3\alpha_0}+4\|C_n(\xi)\|^2_{l^2(0,t_{n^*};L^2)}.
\end{align*}
By $\kappa_{n,i}=\bar\kappa_{n-i}$ with the definition of $\bar\kappa_m$ in (\ref{kap1})--(\ref{kap2}), the discrete Young's convolution inequality and a similar estimate as (\ref{usc}), we have
\begin{align*}
\|C_n(\xi)\|^2_{l^2(0,t_{n^*};L^2)}&\leq Q\tau\sum_{n=1}^{n^*}\Big(\sum_{i=1}^n|\bar\kappa_{n-i}|\|\xi_i\|\Big)^2\\
&\hspace{-0.5in}\leq Q\tau\sum_{n=1}^{n^*}\sum_{i=1}^n|\bar\kappa_{n-i}|\sum_{i=1}^n|\bar\kappa_{n-i}|\|\xi_i\|^2\leq Q\tau\sum_{n=1}^{n^*}\sum_{i=1}^n|\bar\kappa_{n-i}|\|\xi_i\|^2\\
&\hspace{-0.5in}\leq Q\tau\sum_{i=0}^{N-1}|\bar\kappa_{i}|\sum_{i=1}^{n^*}\|\xi_i\|^2\leq Q\|\xi\|^2_{l^2(0,t_{n^*};L^2)}\sum_{i=0}^{N-1}|\bar\kappa_{i}|.
\end{align*}
By (\ref{kap1})--(\ref{kap2}) and (\ref{g'}), we have
\begin{align*}
\sum_{i=0}^{N-1}|\bar\kappa_{i}|\leq \frac{1}{\tau}\int_{-\tau}^{0}\int^{N\tau+z}_{0}e^{-\lambda y}|g'(y)|dydz\leq Q\int_0^Te^{-\lambda y}y^{-\frac{1}{2}}dy\leq Q\lambda^{-\frac{1}{2}}.
\end{align*}
Combining the above three equations lead to
 \begin{align*}
&\|\xi\|^2_{l^2(0,t_{n^*};L^2)}\leq Qh^4+Q\tau^{1+3\alpha_0}+Q\lambda^{-\frac{1}{2}}\|\xi\|^2_{l^2(0,t_{n^*};L^2)}.
\end{align*}
Setting $\lambda$ large enough and choosing $n^*=N$ yield
 \begin{align*}
&\|\xi\|_{l^2(0,T;L^2)}\leq Qh^2+Q\tau^{\frac{1}{2}+\frac{3}{2}\alpha_0}.
\end{align*}
We combine this with $\|\eta_n\|\leq Qh^2$ to get
 \begin{align*}
&\|v-V\|_{l^2(0,T;L^2)}\leq Qh^2+Q\tau^{\frac{1}{2}+\frac{3}{2}\alpha_0},
\end{align*}
and we apply this and $u=e^{\lambda t}v+u_0$ and (\ref{U}) to complete the proof.
\end{proof}

\subsection{Numerical experiments}
\subsubsection{Initial singularity of solutions}
Let $\Omega=(0,1)^2$, $ \alpha(t) = \alpha_0 + 0.1 t$, $u_0=\sin(\pi x)\sin(\pi y)$ and $f=x(1-x)y(1-y)$. As we focus on the initial behavior of the solutions, we take a small terminal time $T=0.1$. We use the uniform spatial partition with the mesh size $h=1/64$ and a very fine temporal mesh size $\tau=1/3000$ to capture the singular behavior of the solutions near $t=0$.
Numerical solutions $U_n(\frac{1}{2},\frac{1}{2})$ defined by  (\ref{U}) is presented in Figure \ref{ComL2} (left) for three cases: 
\begin{equation}\label{aaa}
(\text{i}) ~\alpha_0=0.4; ~(\text{ii}) ~\alpha_0=0.6;~ (\text{iii}) ~\alpha_0=0.8.
\end{equation}
    We find that as $\alpha_0$ decreases, the solutions change more rapidly, indicating a stronger singularity that is consistent with the analysis result  $\p_tu\sim t^{\alpha_0-1}$ in Theorem \ref{thm:reg}.

\subsubsection{Numerical accuracy}  
Let $\Omega=(0,1)^2$, $T=1$ and the exact solution 
    $ u(x) = (1+t^{\alpha_0})\sin(\pi x) \sin (\pi y). $
   The source term $f$ is evaluated accordingly and we select
    $ \alpha(t) = \alpha_0+ 0.1 \sin(2\pi t).$
As the spatial discretization is standard, we only investigate the temporal accuracy of the numerical solution defined by (\ref{U}). We use the uniform spatial partition with the mesh size $h=1/256$ and
 present log-log plots of errors under the $\|\cdot\|_{l^2(0,T;L^2)}$ norm in Figure \ref{ComL2} (right) for different $\alpha_0$ in (\ref{aaa}), which indicates the convergence order of $\frac{1}{2}+\frac{3}{2}\alpha_0$ as proved in Theorem \ref{thm:err}.

\begin{figure}[!hbt]
			\centering
			\begin{minipage}[t]{0.5\linewidth}
				\centering
				\includegraphics[width=2.5in]{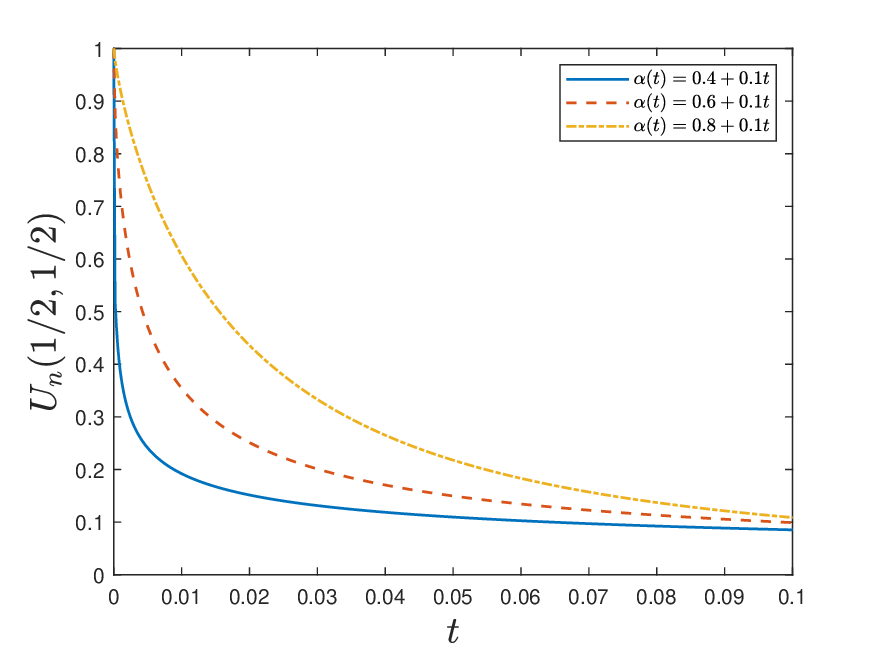}
			\end{minipage}%
			\begin{minipage}[t]{0.5\linewidth}
				\centering
				\includegraphics[width=2.5in]{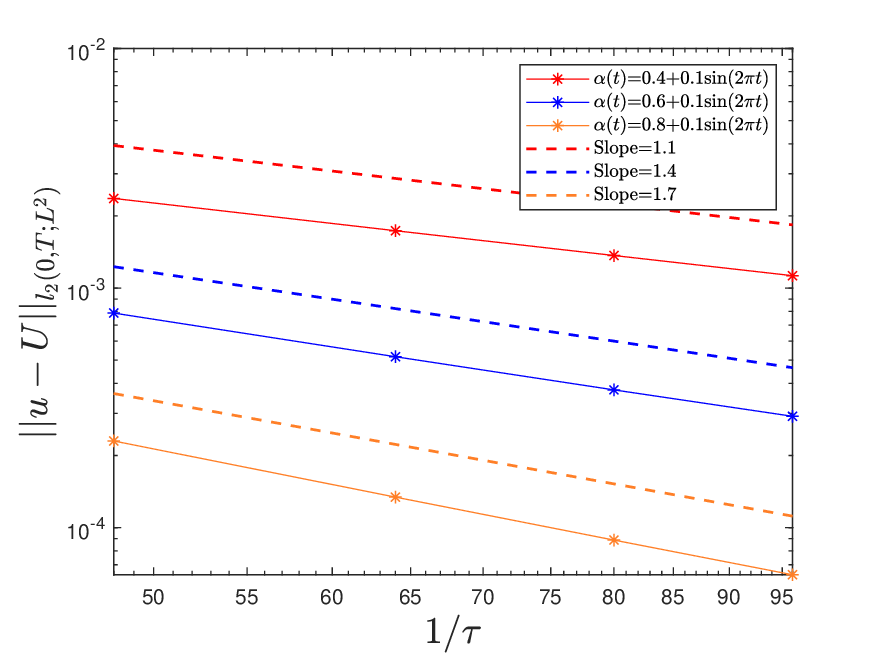}
			\end{minipage}%
			\caption{(left) Plots of $U_n$ at $(\frac{1}{2},\frac{1}{2})$ under different $\alpha_0$; (right) Errors under different $\alpha_0$ and $\tau$.  }
			\label{ComL2}
		\end{figure}

\section{Perturbation method}\label{sec4}
\subsection{Motivation}\label{secc51}
Despite the applicability of the convolution method to fractional initial value problems with variable exponent, it has some restrictions. Below we give two examples that the convolution method does not apply, which motivates an alternative method to analyze these problems.

\textbf{Example 1.}
Consider a two-sided space-fractional diffusion-advection-reaction equation with variable exponent $\alpha(x)\in (1,2)$ \cite{PanPer,SonKar}, which is a variable-exponent extension of the standard space-fractional boundary value problem proposed in \cite{Erv}
\begin{align}\label{bvp}
&-\p_x \big(rI_x^{2-\alpha(x)}+(1-r)\hat I_x^{2-\alpha(x)}\big)\p_x u\\
&\qquad\qquad\qquad+b(x)\p_xu+c(x)u=f(x),~~x\in \Omega:=(0,1);\nonumber\\
&\qquad\qquad u(0)=u(1)=0,\label{bvp2}
\end{align}
where the left and right fractional integrals of variable exponent are defined as
$$ I_x^{2-\alpha(x)}q:=\int_0^x\frac{(x-s)^{1-\alpha(x-s)}}{\Gamma(2-\alpha(x-s))}q(s)ds,~~\hat I_x^{2-\alpha(x)}q:=\int_x^1\frac{(s-x)^{1-\alpha(1-(s-x))}}{\Gamma(2-\alpha(1-(s-x))}q(s)ds. $$
We focus on the two-sided case, i.e. $0<r<1$, since the one-sided case $r=0$ or $r=1$ is simpler.

It is worth mentioning that for $b(x)=c(x)\equiv 0$, $r=1/2$ and $\alpha(x)\equiv \alpha$ for some $1<\alpha<2$, it is shown in \cite[Lemma 2.3]{Aco} that the operator on the left-hand side of (\ref{bvp}) is indeed the fractional Laplacian operator of order $\alpha$. Thus, (\ref{bvp}) could also be viewed as a variable-exponent extension of the fractional Laplacian equation under $b(x)=c(x)\equiv 0$ and $r=1/2$.

Due to the impacts of $\alpha(x)$, it is difficult to apply analytical methods to study (\ref{bvp}), and the convolution method does not apply due to the existence of both left and right fractional integrals of variable exponent. To be specific, although one could employ the convolution method to convert the term $-\p_x I_x^{2-\alpha(x)}\p_xu$ to a more feasible form, the operator $-\p_x \hat I_x^{2-\alpha(x)}\p_x$ is accordingly transformed to a much more complicated Fredholm type operator, which hinders further analysis. 

\textbf{Example 2.} 
Consider a variable-exponent version of the distributed-order diffusion equation, which is firstly proposed in \cite{LorHar} and $\beta(t)$ could serve as a pointwise-in-time shift of the domain of the exponent  
 \begin{align}\label{dis}
&~~~\p_t^{[\mu,\beta]} u(\bm x,t)- \Delta u (\bm x,t)= f(\bm x,t) ,~~(\bm x,t) \in \Omega\times(0,T],\\
\label{3ibc}
& u(\bm x,0)=u_0(\bm x),~\bm x\in \Omega; \quad u(\bm x,t) = 0,~(\bm x,t) \in \p \Omega\times[0,T],
 \end{align}
where the distributed variable-exponent time-fractional derivative is defined as
\begin{align*}
\p_t^{[\mu,\beta]} u:= \p_t (k_{\mu,\beta}*(u-u_0)),~~k_{\mu,\beta}(t):=\int_0^1\frac{t^{-\alpha+\beta(t)}}{\Gamma(1-\alpha+\beta(t))}\mu(\alpha)d\alpha
\end{align*}
for some non-negative continuous function $\mu(\alpha)$ satisfying 
$\text{supp}\,\mu=[b_1,b_2]$
 with $0\leq b_1<b_2<1$ and $\beta(t)$ satisfying $b_2-1+\ve_0\leq  \beta\leq b_1$ over $[0,T]$ for some $0<\ve_0\ll 1$. Without loss of generality, we assume $\beta(0)=0$. Otherwise, one could apply the variable substitution to transfer the initial value of $\beta(t)$ to $0$. Note that when $\beta(t)\equiv 0$, (\ref{dis})--(\ref{3ibc}) is a standard distributed-order diffusion equation. Furthermore, the conditions on $\beta$ ensure that 
\begin{align}\label{ab}
0\leq \alpha-\beta(t)\leq 1-\ve_0,~~\text{ for }\alpha\in [b_1,b_2]\text{ and }t\in [0,T].
\end{align}

Due to the impacts of $\beta(t)$, it is difficult to apply analytical methods to study (\ref{dis}), while the convolution method  does not apply since the existence of the distributed-order integral makes it difficult to find a suitable kernel such that the its convolution with $k_\mu$ generates a generalized identity function.

\subsection{Idea of the method} Motivated by the aforementioned questions, we propose an alternative method called the perturbation method to resolve them. The main idea of the perturbation method is to replace the variable-exponent kernel by a suitable kernel such that their difference serves as a low-order perturbation term, without affecting other terms in the model.

 To illustrate this idea based on the subdiffusion model  (\ref{VtFDEs})--(\ref{ibc}), we split the kernel $k(t)$ into two parts as follows
\begin{align}
k(t)&=\frac{t^{-\alpha(t)}}{\Gamma(1-\alpha(t))}=\frac{t^{-\alpha_0}}{\Gamma(1-\alpha_0)}+\int_0^t\p_z\frac{t^{-\alpha(z)}}{\Gamma(1-\alpha(z))}dz=:\beta_{1-\alpha_0}+\tilde g(t),\label{hh3}
\end{align}
where, by direct calculations and the notation $\gamma(z):=\frac{\Gamma'(z)}{\Gamma(z)}$,
\begin{align*}
\tilde g(t)=\int_0^t\frac{t^{-\alpha(z)}}{\Gamma(1-\alpha(z))}\big(-\alpha'(z)\ln t+\gamma(1-\alpha(z))\alpha'(z) \big)dz.
\end{align*}
As 
$$t^{-\alpha(z)}=t^{-\alpha_0}t^{\alpha_0-\alpha(z)} =t^{-\alpha_0}e^{(\alpha_0-\alpha(z))\ln t}\leq t^{-\alpha_0}e^{Qz|\ln t|}\leq t^{-\alpha_0}e^{Qt|\ln t|}\leq Qt^{-\alpha_0},$$
$\tilde g$ could be bounded as
\begin{align}\label{bndg}
|\tilde g|\leq Q\int_0^t t^{-\alpha_0}(1+|\ln t|)dz\leq Qt^{1-\alpha_0}(1+|\ln t|)\rightarrow 0\text{ as }t\rightarrow 0^+. 
\end{align}
By similar estimates we bound $\tilde g'$ as
\begin{align}\label{bndg'}
|\tilde g'|\leq Qt^{-\alpha_0}(1+|\ln t|)\in L^1(0,T). 
\end{align}
We then invoke (\ref{hh3}) in (\ref{VtFDEs}) to get
\begin{align}\label{qq1}
^c\p_t^{\alpha_0}u+\tilde g*\p_tu -\Delta u=f. 
\end{align}
We apply the integration by parts to get
$$\tilde g*\p_tu=-\tilde gu_0+\tilde g'*u, $$
which, together with (\ref{qq1}), leads to
\begin{align*}
^c\p_t^{\alpha_0}u+\tilde g'*u -\Delta u=f+\tilde gu_0. 
\end{align*}
As the term $\tilde g'*u$ is a low-order term compared with $^c\p_t^{\alpha_0}u$, one could analyze this model by means of Laplace transform as we did in Section \ref{seccm}.

\noindent\underline{\textbf{Relation between two methods}}

Note that the transformed model (\ref{qq1}) looks different from (\ref{Model2}). However, if we further apply $\p_t^{1-\alpha_0}$ on both sides of (\ref{qq1}), we have
\begin{align}\label{qq3}
\p_tu+\p_t^{1-\alpha_0}(\tilde g*\p_tu) -\p_t^{1-\alpha_0}\Delta u=\p_t^{1-\alpha_0}f. 
\end{align}
Now the two transformed models (\ref{qq3}) and
(\ref{Model2}) are exactly the same equation if we can show that
$$\p_t^{1-\alpha_0}(\tilde g*\p_tu)=g'*\p_tu. $$
To prove this relation, recall that $g=\beta_{\alpha_0}*k$ with $g(0)-1=0$ and $\tilde g=k-\beta_{1-\alpha_0}$. Then direct calculations show that
$$\p_t^{1-\alpha_0}(\tilde g*\p_tu)=\p_t\big[\beta_{\alpha_0}*(k-\beta_{1-\alpha_0})*\p_tu\big]=\p_t\big[(g-1)*\p_tu\big]=g'*\p_tu. $$
Thus, for models such as the variable-exponent subdiffusion equation (\ref{VtFDEs}), both the convolution method and the perturbation method could transform the original problem to the same model such that the analysis such as those in Section \ref{seccm} could be performed. Nevertheless, the perturbation method applies to more complicated models such as those in \textbf{Examples 1-2} , as we will see later.
 
\subsection{Space-fractional problem with variable exponent}
We consider the two-sided space-fractional diffusion-advection-reaction equation (\ref{bvp})--(\ref{bvp2}) proposed in \textbf{Example 1}. It is worth mentioning that if we denote the equation (\ref{bvp}) as $\mathcal L_{x,r}^{\alpha(t)}u=f$, the upcoming analysis also applies to the multi-dimensional case, e,g, the two-dimensional case $\mathcal L_{x,r}^{\alpha(x)}u+\mathcal L_{y,\bar r}^{\bar\alpha(x)}u=f$ for some $0<\bar r<1$ and $1<\bar\alpha(x)<2$ on $(x,y)\in (0,1)^2$.

Recall that $\alpha_0=\alpha(0)$ and  $\alpha_1=\alpha(1)$.
 Following the idea of the perturbation method, we split the kernels in $I_x^{2-\alpha(x)}$ and $\hat I_x^{2-\alpha(x)}$ as
$$\frac{x^{1-\alpha(x)}}{\Gamma(2-\alpha(x))}=\beta_{2-\alpha_0}+g_l(x),~~g_l(x):=\int_0^x\p_z\frac{x^{1-\alpha(z)}}{\Gamma(2-\alpha(z))}dz, $$
and
$$\frac{x^{1-\alpha(1-x)}}{\Gamma(2-\alpha(1-x))}=\beta_{2-\alpha_1}+g_r(x),~~g_r(x):=\int_0^x\p_z\frac{x^{1-\alpha(1-z)}}{\Gamma(2-\alpha(1-z))}dz. $$
By similar estimates as (\ref{bndg})--(\ref{bndg'}), we have the following estimates for $0<\ve\ll 1$
\begin{align}
&g_l(0)=g_r(0)=0,~~|g_l'|\leq Q_1x^{1-\alpha_0-\ve},~~|g_r'|\leq Q_1x^{1-\alpha_1-\ve},\nonumber\\
&\hspace{1.18in}|g_l''|\leq Q_2x^{-\alpha_0-\ve},~~|g_r''|\leq Q_2x^{-\alpha_1-\ve}.\label{glr}
\end{align}
Then equation (\ref{bvp}) is equivalently rewritten as
\begin{align*}
&-\p_x \big(rI_x^{2-\alpha_0}+(1-r)\hat I_x^{2-\alpha_1}\big)\p_x u\\
&\quad-r \int_0^xg_l'(x-s)\p_su(s)ds+(1-r)\int_x^1g_r'(s-x)\p_s u(s)ds\\
&\quad+b(x)\p_xu+c(x)u=f(x).
\end{align*}
By integration by parts, we define the corresponding bilinear form by
\begin{align*}
a(u,v):=& r\big(I_x^{1-\frac{\alpha_0}{2}}\p_x u,\hat I_x^{1-\frac{\alpha_0}{2}}\p_x v\big)+(1-r)\big(I_x^{1-\frac{\alpha_1}{2}}\p_x u,\hat I_x^{1-\frac{\alpha_1}{2}}\p_x v\big)\\
&\quad-r( \int_0^xg_l'(x-s)\p_su(s)ds,v)+(1-r)(\int_x^1g_r'(s-x)\p_s u(s)ds,v)\\
&\quad+(b(x)\p_xu,v)+(c(x)u,v).
\end{align*}
Then the variational formulation of (\ref{bvp})--(\ref{bvp2}) reads: find $u\in H^{\alpha^*}_0$ where $\alpha^*=\frac{\max\{\alpha_0,\alpha_1\}}{2}$ such that
\begin{align}\label{varia}
a(u,v)=(f,v),~~\forall v\in H^{\alpha^*}_0.
\end{align}

The following estimates are critical in both mathematical and numerical analysis. 
\begin{lemma}\label{lemcoer}
Suppose $b\in C^1[0,1]$ and $c$ is bounded. Then there exist a positive constant $\bar c$ such that when
\begin{align*}
c-\frac{1}{2}b'\geq \bar c,\text{ for }x\in [0,1],
\end{align*}
 the following properties hold for any $u,v\in H^{\alpha^*}_0$
\begin{align*}
a(u,u)\geq c_*\|u\|_{H^{\alpha^*}}^2,~~|a(u,v)|\leq c^*\|u\|_{H^{\alpha^*}}\|v\|_{H^{\alpha^*}},
\end{align*}
for some positive constants $c_*$ and $c^*$.
\end{lemma}
\begin{proof}
By the Sobolev embedding theorem, we have $u,v\in C[0,1]$. To show the coercivity, we apply (\ref{er}) to get
\begin{align}
& r\big(I_x^{1-\frac{\alpha_0}{2}}\p_x u,\hat I_x^{1-\frac{\alpha_0}{2}}\p_x u\big)+(1-r)\big(I_x^{1-\frac{\alpha_1}{2}}\p_x u,\hat I_x^{1-\frac{\alpha_1}{2}}\p_x u\big)\nonumber\\
 &\quad\geq rc_1\|u\|_{H^{\frac{\alpha_0}{2}}}^2+(1-r)c_1\|u\|_{H^{\frac{\alpha_1}{2}}}^2.\label{coer1}
\end{align}
Then an integration by parts yields
\begin{align}\label{coer2}
(b(x)\p_xu,u)+(c(x)u,u)=\big((c-\frac{1}{2}b')u,u\big). 
\end{align}
The estimate of the term containing $g_l$ requires technical derivations. By (\ref{er0}),direct calculations yield
\begin{align*}
&( \int_0^xg_l'(x-s)\p_su(s)ds,u)\\
&\quad=\int_0^1\int_0^xg_l'(x-s)\p_su(s)dsu(x) dx\\
&\quad=\int_0^1\p_su(s)\int_s^1g_l'(x-s)u(x) dxds\\
&\quad=\int_0^1\p_s(I^1_s\p_su(s))\int_s^1g_l'(x-s)u(x) dxds\\
&\quad=-\int_0^1\big(I^1_s\p_su(s)\big)\p_s\int_s^1g_l'(x-s)u(x) dxds\\
&\quad=-\int_0^1\big(I^{1-\frac{\alpha_0}{2}}_s\p_su(s)\big)\hat I^{\frac{\alpha_0}{2}}_s\big(\p_s\int_s^1g_l'(x-s)u(x) dx\big)ds.
\end{align*}
By the relation $\p_s\hat I^{\frac{\alpha_0}{2}}_sq=\hat I^{\frac{\alpha_0}{2}}_s\p_sq$ when $q(1)=0$ \cite{Jinbook}, we have
\begin{align}
&( \int_0^xg_l'(x-s)\p_su(s)ds,u)\nonumber\\
&\quad=-\int_0^1\big(I^{1-\frac{\alpha_0}{2}}_s\p_su(s)\big)\p_s\big(\hat I^{\frac{\alpha_0}{2}}_s\int_s^1g_l'(x-s)u(x) dx\big)ds.\label{bvp5}
\end{align}
By variable substitution $z=\frac{y-s}{x-s}$ we have
\begin{align}
&\hat I^{\frac{\alpha_0}{2}}_s\int_s^1g_l'(x-s)u(x) dx=\int_s^1\frac{(y-s)^{\frac{\alpha_0}{2}-1}}{\Gamma(\frac{\alpha_0}{2})}\int_y^1g_l'(x-y)u(x) dxdy\nonumber\\
&\quad=\int_s^1u(x)\int_s^x \frac{(y-s)^{\frac{\alpha_0}{2}-1}}{\Gamma(\frac{\alpha_0}{2})}g_l'(x-y)dydx\label{bvp6}\\
&\quad=\int_s^1u(x)(x-s)^{\frac{\alpha_0}{2}}\int_0^1 \frac{z^{\frac{\alpha_0}{2}-1}}{\Gamma(\frac{\alpha_0}{2})}g_l'((x-s)(1-z))dzdx.\nonumber
\end{align}
By (\ref{glr}) we have
\begin{align*}
&\Big|u(x)(x-s)^{\frac{\alpha_0}{2}}\int_0^1 \frac{z^{\frac{\alpha_0}{2}-1}}{\Gamma(\frac{\alpha_0}{2})}g_l'((x-s)(1-z))dz \Big|\\
&\quad\leq Q(x-s)^{1-\frac{\alpha_0}{2}-\ve}\int_0^1 \frac{z^{\frac{\alpha_0}{2}-1}}{\Gamma(\frac{\alpha_0}{2})}(1-z)^{1-\alpha_0-\ve}dz\rightarrow 0 \text{ as }x\rightarrow s^+,
\end{align*}
then we invoke (\ref{bvp6}) to get
\begin{align*}
&\Big|\p_s\big(\hat I^{\frac{\alpha_0}{2}}_s\int_s^1g_l'(x-s)u(x) dx\big)\Big|\\
&\quad =\Big|\p_s \int_s^1u(x)(x-s)^{\frac{\alpha_0}{2}}\int_0^1 \frac{z^{\frac{\alpha_0}{2}-1}}{\Gamma(\frac{\alpha_0}{2})}g_l'((x-s)(1-z))dzdx\Big|\\
&\quad =\Big|-\frac{\alpha_0}{2} \int_s^1u(x)(x-s)^{\frac{\alpha_0}{2}-1}\int_0^1 \frac{z^{\frac{\alpha_0}{2}-1}}{\Gamma(\frac{\alpha_0}{2})}g_l'((x-s)(1-z))dzdx\\
&\qquad-\int_s^1u(x)(x-s)^{\frac{\alpha_0}{2}}\int_0^1 \frac{z^{\frac{\alpha_0}{2}-1}}{\Gamma(\frac{\alpha_0}{2})}g_l''((x-s)(1-z))(1-z)dzdx\Big|\\
&\quad\leq Q_1\frac{\alpha_0}{2} \int_s^1|u(x)|(x-s)^{-\frac{\alpha_0}{2}-\ve}\int_0^1 \frac{z^{\frac{\alpha_0}{2}-1}}{\Gamma(\frac{\alpha_0}{2})}(1-z)^{1-\alpha_0-\ve}dzdx\\
&\qquad+Q_2\int_s^1|u(x)|(x-s)^{\frac{-\alpha_0}{2}-\ve}\int_0^1 \frac{z^{\frac{\alpha_0}{2}-1}}{\Gamma(\frac{\alpha_0}{2})}(1-z)^{1-\alpha_0-\ve}dzdx\Big|\\
&\quad\leq \frac{Q_3}{\Gamma(1-\frac{\alpha_0}{2}-\ve)} \int_s^1|u(x)|(x-s)^{-\frac{\alpha_0}{2}-\ve}dx= Q_3\hat I_s^{1-\frac{\alpha_0}{2}-\ve}|u|,
\end{align*}
which, together with the fact that $\hat I_s^{1-\frac{\alpha_0}{2}-\ve}$ is a bounded linear operator in $L^2$, leads to
\begin{align*}
&\Big\|\p_s\big(\hat I^{\frac{\alpha_0}{2}}_s\int_s^1g_l'(x-s)u(x) dx\big)\Big\|\leq Q_3\|\hat I_s^{1-\frac{\alpha_0}{2}-\ve}|u|\|\leq Q_4\|u\|.
\end{align*}
We invoke this in (\ref{bvp5}) to finally obtain for $\ve_1>0$
\begin{align}
&\Big|( \int_0^xg_l'(x-s)\p_su(s)ds,u)\Big|\nonumber\\
&\qquad\leq Q_4\|\p_x^{\frac{\alpha_0}{2}}u\|\|u\|\leq \frac{Q_4}{c_3}\|u\|_{H^{\frac{\alpha_0}{2}}}\|u\|\leq \ve_1\|u\|_{H^{\frac{\alpha_0}{2}}}^2+\frac{Q_4^2}{4c_3^2\ve_1}\|u\|^2. \label{bvp7}
\end{align}
By similar estimates, we have the analogous result
\begin{align}
&\Big|(\int_x^1g_r'(s-x)\p_s u(s)ds,u)\Big|\nonumber\\
&\qquad\leq Q_5\|\hat\p_x^{\frac{\alpha_1}{2}}u\|\|u\|\leq \frac{Q_5}{c_3}\|u\|_{H^{\frac{\alpha_1}{2}}}\|u\|\leq \ve_1\|u\|_{H^{\frac{\alpha_1}{2}}}^2+\frac{Q_5^2}{4c_3^2\ve_1}\|u\|^2. \label{bvp8}
\end{align}
Combining (\ref{coer1}), (\ref{coer2}), (\ref{bvp7}) and (\ref{bvp8}) leads to
\begin{align*}
&a(u,u)\geq r(c_1-\ve_1)\|u\|_{H^{\frac{\alpha_0}{2}}}^2+(1-r)(c_1-\ve_1)\|u\|_{H^{\frac{\alpha_1}{2}}}^2\\
&\quad+\big((c-\frac{1}{2}b'-r\frac{Q_4^2}{4c_3^2\ve_1}-(1-r)\frac{Q_5^2}{4c_3^2\ve_1})u,u\big).
\end{align*}
Thus, for $\ve_1<c_1$ and 
$$c-\frac{1}{2}b'\geq r\frac{Q_4^2}{4c_3^2\ve_1}+(1-r)\frac{Q_5^2}{4c_3^2\ve_1},$$
we reach the coercivity. 

The proof of the second statement follows from (\ref{er}) and the above estimates and is thus omitted. The only issue that needs to be pointed out is that  the condition $b\in C^1[0,1]$ is used to ensure $\|bv\|_{H^{\alpha^*}}\leq Q\|v\|_{H^{\alpha^*}}$
by \cite[Lemma 3.2]{Erv} such that the term $(b\p_xu,v)$ could be bounded by $Q\|u\|_{H^{\alpha^*}}\|v\|_{H^{\alpha^*}}$, see e.g. \cite[Equation 28]{Erv}.
\end{proof}

Based on Lemma \ref{lemcoer} and the estimate for $f\in H^{-\alpha^*}$
$$|(f,v)|\leq \|f\|_{H^{-\alpha^*}}\|v\|_{H^{\alpha^*}}, $$
which means that $F(v):=(f,v)$ is a continuous linear functional over $H^{\alpha^*}$, we apply the Lax-Milgram theorem to reach the following conclusion.
\begin{theorem}
Under the conditions of Lemma \ref{lemcoer} and  $f\in H^{-\alpha^*}$, there exists a unique solution to (\ref{varia}) in $H^{\alpha^*}_0$ such that
$$\|u\|_{H^{\alpha^*}}\leq Q \|f\|_{H^{-\alpha^*}}.$$
\end{theorem}

\subsection{Distributed variable-exponent model}
 We employ the idea of the perturbation method  to analyze the distributed variable-exponent model (\ref{dis})--(\ref{3ibc}) in \textbf{Example 2}. We apply $\beta(0)=0$ to split $k_\mu(t)$ into two parts as follows
\begin{align}
k_{\mu,\beta}(t)&=\int_0^1\frac{t^{-\alpha}}{\Gamma(1-\alpha)}\mu(\alpha)d\alpha+\int_0^1\int_0^t\p_z\frac{t^{-\alpha+\beta(z)}}{\Gamma(1-\alpha+\beta(z))}dz\mu(\alpha)d\alpha\nonumber\\
&=:k_\mu(t)+\int_0^1g(t;\alpha)\mu(\alpha)d\alpha,~~g(t;\alpha):=\int_0^t\p_z\frac{t^{-\alpha+\beta(z)}}{\Gamma(1-\alpha+\beta(z))}dz.\label{dis3}
\end{align}
Based on (\ref{ab}), direct calculations show that 
\begin{align}
|g(t;\alpha)|&=\Big|\int_0^t \frac{t^{-\alpha+\beta(z)}}{\Gamma(1-\alpha+\beta(z))}\Big(\beta'(z)\ln t-\frac{\beta'(z)}{\Gamma(1-\alpha+\beta(z))}\Big)dz\Big|\nonumber\\
&\leq Qt^{1-\alpha+\beta(z)}(|\ln t|+1)\leq Qt^{\ve_0}(|\ln t|+1);~~\lim_{t\rightarrow 0^+}g(t;\alpha)=0;\label{dis5}\\
|\p_tg(t;\alpha)|&= \Big|\frac{t^{-\alpha+\beta(t)}}{\Gamma(1-\alpha+\beta(t))}\Big(\beta'(t)\ln t-\frac{\beta'(t)}{\Gamma(1-\alpha+\beta(t))}\Big)\nonumber\\
&\qquad+\int_0^t \p_t\Big[\frac{t^{-\alpha+\beta(z)}}{\Gamma(1-\alpha+\beta(z))}\Big(\beta'(z)\ln t-\frac{\beta'(z)}{\Gamma(1-\alpha+\beta(z))}\Big)\Big]dz\Big|\nonumber\\
&\leq Qt^{-(1-\ve_0)}(|\ln t|+1)\leq Qt^{-(1-\frac{\ve_0}{2})}.\label{dis6}
\end{align}
We then invoke (\ref{dis3}) in (\ref{dis}) and apply (\ref{dis5}) to get
 \begin{align}\label{dis4}
&~~~\p_t^{[\mu]} u+ \int_0^1\p_tg(t;\alpha)\mu(\alpha)d\alpha*(u-u_0)- \Delta u = f,
 \end{align}
where $\p_t^{[\mu]} $ is the standard distributed-order operator defined as \cite{KubRys,LorHar}
$$\p_t^{[\mu]} u:= \p_t (k_{\mu}*(u-u_0)).$$
Then we follow \cite[Section 3]{JinKia} to take the Laplace transform of (\ref{dis4}) to represent $u$ as
\begin{align}\label{disu}
u=S_1(t)u_0+\int_0^tS_2(t-s)\Big(- \int_0^1\p_sg(s;\alpha)\mu(\alpha)d\alpha*(u-u_0)+f(s)\Big)ds.
\end{align}
Here $S_1$ and $S_2$ are given as  
\begin{align*}
S_1(t)q&:=\frac{1}{2\pi \rm i}\int_{\Gamma_{\theta_1}} e^{tz}(\mathcal V(z)-\Delta)^{-1}z^{-1}\mathcal V(z)qd z,\\
S_2(t)q&:=\frac{1}{2\pi \rm i}\int_{\Gamma_{\theta_1}} e^{tz}(\mathcal V(z)-\Delta)^{-1}qd z,
\end{align*}
where $\mathcal V(z):=\int_0^1z^{\alpha}\mu(\alpha)d\alpha$ and $\theta_1\in (0,\frac{2\theta-\pi}{8})$ with $\theta\in (\frac{\pi}{2},\min\{\frac{\pi}{2b_2},\pi\})$, with the following estimates \cite[Lemma 3.1]{JinKia}
\begin{align}\label{dis7}
\|S_1(t)\|\leq Qt^{\frac{b_1+b_2}{2}-1},~~\|S_2(t)\|\leq Qt^{\frac{b_1-b_2}{2}}.
\end{align}
Then we follow \cite{JinKia} to define a weak solution to (\ref{dis})--(\ref{3ibc}): a function $u\in L^1(L^2)$ is a weak solution to (\ref{dis})--(\ref{3ibc}) if it satisfies (\ref{disu}). The following theorem gives the existence and uniqueness of the weak solution to (\ref{dis})--(\ref{3ibc}).
\begin{theorem}
Suppose $u_0\in L^2$ and $f\in L^1(L^2)$, then model (\ref{dis})--(\ref{3ibc}) admits a unique weak solution $u\in L^1(L^2)$ with the following estimate
$$\|u\|_{L^1(L^2)}\leq Q(\|u_0\|+\|f\|_{L^1(L^2)}). $$
\end{theorem}
\begin{proof}
Define a mapping $\mathcal M:\,L^1(L^2)\rightarrow L^1(L^2) $ by $v=\mathcal Mw$ where $v$ satisfies 
\begin{align}\label{dis9}
v=S_1(t)u_0+\int_0^tS_2(t-s)\Big(- \int_0^1\p_sg(s;\alpha)\mu(\alpha)d\alpha*(w-u_0)+f(s)\Big)ds.
\end{align}
For $\lambda\geq 0$, the equivalent norm $\|\cdot\|_{L^1_\lambda(L^2)}$ of $L^1(L^2)$ is given as $\|q\|_{L^1_\lambda(L^2)}:=\|e^{-\lambda t}q\|_{_{L^1(L^2)}}$. Then we first show that $\mathcal M$ is well defined. For $w\in L^1(L^2)$, we apply the $\|\cdot\|_{L^1_\lambda(L^2)}$ norm on both sides of (\ref{dis9}) and use (\ref{dis7}) and (\ref{dis6}) to get
\begin{align}
&\|v\|_{L^1_\lambda(L^2)}\leq \|S_1(t)u_0\|_{L^1_\lambda(L^2)}\nonumber\\
&\qquad+\Big\|\int_0^tS_2(t-s)\Big(- \int_0^1\p_sg(s;\alpha)\mu(\alpha)d\alpha*(w-u_0)+f(s)\Big)ds\Big\|_{L^1_\lambda(L^2)}\nonumber\\
&~~\leq Q\|t^{\frac{b_1+b_2}{2}-1}\|_{L^1(0,T)}\|u_0\|+Q\big\|t^{\frac{b_1-b_2}{2}}*\big(t^{-(1-\frac{\ve_0}{2})}*\|w-u_0\|+\|f\|  \big)\big\|_{L^1_\lambda(0,T)}\nonumber\\
&~~\leq Q(\|u_0\|+\|f\|_{L^1(L^2)})+Q\|e^{-\lambda t}t^{\frac{b_1-b_2}{2}}\|_{L^1(0,T)}\|w\|_{L^1_\lambda(L^2)}\nonumber\\
&~~\leq Q(\|u_0\|+\|f\|_{L^1(L^2)})+Q\sigma^{\frac{b_2-b_1}{2}-1}\|w\|_{L^1_\lambda(L^2)},\label{dis10}
\end{align}
where we applied a similar estimate as (\ref{zxc8})
\begin{align*}
 \|e^{-\sigma t}t^{\frac{b_1-b_2}{2}}\|_{L^1(0,T)}=\int_0^Te^{-\sigma t}t^{\frac{b_1-b_2}{2}}dt\leq \sigma^{\frac{b_2-b_1}{2}-1}\int_0^\infty e^{- t}t^{\frac{b_1-b_2}{2}}dt\leq Q\sigma^{\frac{b_2-b_1}{2}-1}. 
 \end{align*}
 Thus, $v\in L^1(L^2)$ such that $\mathcal M$ is well defined. To show the contractivity, it suffices to consider (\ref{dis9}) with $w=v$ and $u_0=f=0$. Then (\ref{dis10}) gives 
 \begin{align*}
&\|v\|_{L^1_\lambda(L^2)}\leq Q\sigma^{\frac{b_2-b_1}{2}-1}\|v\|_{L^1_\lambda(L^2)}.
\end{align*}
 Choosing $\sigma$ large enough ensures $Q\sigma^{\frac{b_2-b_1}{2}-1}<1$, which implies that $\mathcal M$ is a contraction mapping such that there exists a unique solution $u$ in $L^1(L^2)$ to (\ref{disu}). The stability estimate follows from (\ref{dis10}) with $v=w=u$ and $\sigma$ large enough. The proof is thus completed.
\end{proof}

\section{Application to Abel integral equation}
We further demonstrate the usage of the perturbation method by analyzing the following Abel integral equation of variable exponent, a typical first-kind Volterra integral equation with weak singularity that  would broaden the usefulness of its constant-exponent version as commented in \cite{LiaSty}
 \begin{align}\label{vie}
& (k* u)(t)= f(t) ,~~t\in [0,T]. 
 \end{align}
 \subsection{The case $0<\alpha_0<1$}\label{sec71}
 We prove the well-posedness of (\ref{vie}) in the following theorem.
 \begin{theorem}
 Suppose $f\in W^{1,1}(0,T)$. Then the integral equation (\ref{vie}) admits a unique solution in $L^1(0,T)$ such that
  $$\|u\|_{L^1(0,T)}\leq Q\|f\|_{W^{1,1}(0,T)}. $$
  If further $f\in W^{1,\frac{1}{\alpha_0-\sigma}}(0,T)$ for $0<\sigma\ll \alpha_0$, then 
$$|u|\leq Q|f(0)|t^{\alpha_0-1}+Q\|f'\|_{L^{\frac{1}{\alpha_0-\sigma}}(0,T)},~~t\in (0,T]. $$

 \end{theorem}
 \begin{proof}
By the idea of the perturbation method, we split $k$ as (\ref{hh3}) such that the integral equation (\ref{vie}) could be equivalently rewritten as
\begin{align}\label{vie15}
\beta_{1-\alpha_0}* u=-\tilde g*u+ f.
\end{align}
As $\p_t^{1-\alpha_0}$ is the inverse operator of $I_t^{1-\alpha_0}$, we apply $\p_t^{1-\alpha_0}$ on both sides of this equation to get
\begin{align}\label{vie1}
 u=-\p_t^{1-\alpha_0}(\tilde g*u)+\p_t^{1-\alpha_0} f=-\big[(\beta_{\alpha_0}*\tilde g)*u\big]'+\p_t^{1-\alpha_0} f.
\end{align}
As $\tilde g$ is bounded (see (\ref{bndg})), we have $|\beta_{\alpha_0}*\tilde g|\rightarrow 0$ as $t\rightarrow 0^+$ such that (\ref{vie2}) could be further rewritten as a second-kind Volterra integral equation
\begin{align}\label{vie2}
 u=-(\beta_{\alpha_0}*\tilde g)'*u+\p_t^{1-\alpha_0} f.
\end{align}
By $\tilde g(0)=0$ and (\ref{bndg'}), we have
\begin{align}\label{vie7}
|(\beta_{\alpha_0}*\tilde g)'|=|\beta_{\alpha_0}*\tilde g'|\leq Q\beta_{\alpha_0}*t^{-\alpha_0-\ve}\leq Qt^{-\ve},
\end{align}
which means that $(\beta_{\alpha_0}*\tilde g)'\in L^1(0,T)$. Furthermore, for $f\in W^{1,1}(0,T)$, we have
\begin{align}\label{vie8}
\p_t^{1-\alpha_0} f=(\beta_{\alpha_0}*f)'=\beta_{\alpha_0}f(0)+\beta_{\alpha_0}*f'\in L^1(0,T).
\end{align}
Thus, by classical results on the
second-kind Volterra integral equation, see e.g. \cite[Theorem 2.3.5]{Gri}, model (\ref{vie2}) admits a unique solution in $L^1(0,T)$. 
 As the convolution of $\beta_{1-\alpha_0}$ and (\ref{vie2}) leads to (\ref{vie1}) and thus (\ref{vie}), the original model (\ref{vie}) has an $L^1$ solution. The uniqueness of the $L^1$ solution of (\ref{vie}) follows from that of (\ref{vie2}).

To derive the first stability estimate, we multiply (\ref{vie2}) by $e^{-\lambda t}$ for some $\lambda>0$ to get
\begin{align*}
 e^{-\lambda t}u=-[e^{-\lambda t}(\beta_{\alpha_0}*\tilde g)']*[e^{-\lambda t}u]+e^{-\lambda t}\p_t^{1-\alpha_0} f.
\end{align*}
Apply the $L^1$ norm on both sides of this equation and employ (\ref{vie7})--(\ref{vie8}) and the Young's convolution inequality to get
\begin{align}
\| e^{-\lambda t}u\|_{L^1(0,T)}&\leq Q\|e^{-\lambda t}t^{-\ve}\|_{L^1(0,T)}\|e^{-\lambda t}u\|_{L^1(0,T)}+Q\|f\|_{W^{1,1}(0,T)}\nonumber\\
&\leq Q\lambda^{\ve-1}\|e^{-\lambda t}u\|_{L^1(0,T)}+Q\|f\|_{W^{1,1}(0,T)}\label{vies1}
\end{align}
where we use a similar estimate as (\ref{zxc8}) to bound  $\|e^{-\lambda t}t^{-\ve}\|_{L^1(0,T)}$. Then we choose $\lambda$ large enough to get the first stability result.

To obtain the second estimate of this theorem, (\ref{vie2}), (\ref{vie7}) and (\ref{vie8}) imply
$$|u|\leq QI_t^{1-\ve}|u|+\beta_{\alpha_0}|f(0)|+\beta_{\alpha_0}*|f'|. $$
We apply the Young's convolution inequality with $\frac{1}{\frac{1}{1-\alpha_0+\sigma}}+\frac{1}{\frac{1}{\alpha_0-\sigma}}=1$ for $0<\sigma\ll \alpha_0$ to get
$$\|\beta_{\alpha_0}*|f'|\|_{L^\infty(0,t)}\leq Qt^{\sigma}\|f'\|_{L^{\frac{1}{\alpha_0-\sigma}}(0,T)}. $$
Combine the above two equations leads to
$$|u|\leq QI_t^{1-\ve}|u|+\beta_{\alpha_0}|f(0)|+Q\|f'\|_{L^{\frac{1}{\alpha_0-\sigma}}(0,T)}. $$
Then an application of the weakly singular Gronwall inequality, see e.g. \cite[Theorem 4.2]{Jinbook}, leads to the second estimate of this theorem and thus completes the proof.
\end{proof}

\subsection{The case $\alpha_0=0$}\label{sec72}
We consider the special case $\alpha_0=0$. If $\alpha_0=0$, then $\beta_{1-\alpha_0}=1$ such that (\ref{vie15}) becomes
\begin{align*}
\int_0^t u(s)ds=-\tilde g*u+ f.
\end{align*}
Differentiate this equation and use the estimate (\ref{bndg}) lead to
\begin{align}\label{vies2}
u=-\tilde g'*u+ f'.
\end{align}
  \begin{theorem}\label{thm72}
 Suppose $f\in W^{1,1}(0,T)$, $\alpha_0=0$. Then if $f(0)=0$, the integral equation (\ref{vie}) admits a unique solution in $L^1(0,T)$ such that
  $$\|u\|_{L^1(0,T)}\leq Q\|f\|_{W^{1,1}(0,T)}. $$
  If further $|f'|\leq L_f$ for $t\in [0,T]$ for some constant $L_f>0$, then 
$$|u|\leq QL_f,~~t\in [0,T]. $$

 \end{theorem}
\begin{proof}
By (\ref{bndg'}), $\tilde g'\in L^1(0,T)$ such that we apply similar and simpler estimates as (\ref{vie})--(\ref{vies1}) to reach that (\ref{vies2}) admits a unique $L^1$ solution with the estimate 
$$\|u\|_{L^1(0,T)}\leq Q\|f'\|_{L^1(0,T)}. $$
Then we integrate (\ref{vies2}) to obtain
\begin{align*}
\int_0^tu(s)ds&=-\int_0^t \tilde g'*uds+ \int_0^tf'(s)ds\\
&= -\int_0^t (\tilde g*u)'(s)ds+ \int_0^tf'(s)ds=\tilde g*u+f-f(0).
\end{align*}
Thus, the condition $f(0)=0$ implies that the original problem (\ref{vie}) has an $L^1$ solution. The uniqueness follows from that of (\ref{vies2}). The second estimate follows immediately from the weakly singular Gronwall inequality.
\end{proof}
\subsection{The case $\alpha_0=1$} 
For this case, we further assume $\alpha'(0)\neq 0$ to simplify the derivations and highlight the main ideas. We apply the idea of the perturbation method to give a different splitting of $k(t)$ from (\ref{hh3}) as follows
$$k(t)=\lim_{t\rightarrow 0^+}k(t)+\bar g.  $$
Direct calculations show that 
\begin{align}\label{m3as1}
\lim_{t\rightarrow 0^+}k(t)=\lim_{t\rightarrow 0^+}\frac{t^{-\alpha(t)}}{\Gamma(1-\alpha(t))}=\lim_{t\rightarrow 0^+}\frac{1}{\Gamma(2-\alpha(t))}\frac{1-\alpha(t)}{t^{\alpha(t)}}.
\end{align}
As $\lim_{t\rightarrow 0^+}|(1-\alpha(t))\ln t|\leq \lim_{t\rightarrow 0^+}|Qt\ln t|=0$, we have 
$$\lim_{t\rightarrow 0^+}t^{\alpha(t)-1}=\lim_{t\rightarrow 0^+}e^{(\alpha(t)-1)\ln t}=1$$
 such that
\begin{align}\label{lim}
\lim_{t\rightarrow 0^+}\frac{1-\alpha(t)}{t^{\alpha(t)}}=\lim_{t\rightarrow 0^+}\frac{-\alpha'(t)}{t^{\alpha(t)}(\alpha'(t)\ln t+\frac{\alpha(t)}{t})}=- \alpha'(0).
\end{align}
Combine the above three estimates to get
$$k(t)=-\alpha'(0)+\bar g,~~\bar g:=k(t)+\alpha'(0)\text{ such that }\bar g(0)=0. $$
We thus rewrite (\ref{vie}) as
\begin{align*}
-\alpha'(0)\int_0^tu(s)ds=-\bar g*u+f.
\end{align*}
By (\ref{lim}), $\bar g$ is bounded. Then we apply the expression of $k$ in (\ref{m3as1}) to evaluate $\bar g'=k'$ as
\begin{align*}
\bar g'=\Big(\frac{1}{\Gamma(2-\alpha(t))}\Big)'\frac{1-\alpha(t)}{t^{\alpha(t)}}+\frac{1}{\Gamma(2-\alpha(t))}\frac{-\alpha'(t)-(1-\alpha(t))(\alpha'(t)\ln t+\frac{\alpha(t)}{t})}{t^{\alpha(t)}}.
\end{align*}
By (\ref{lim}), the first right-hand side term is bounded. For the second right-hand side term, we rewrite the numerator as
$$-\alpha'(t)(1-\alpha(t))(1+\ln t)+\alpha(t)(-\alpha'(t)-\frac{1-\alpha(t)}{t}). $$
By (\ref{lim}), we have
$$\Big|\frac{-\alpha'(t)(1-\alpha(t))(1+\ln t)}{t^{\alpha(t)}}\Big|\leq Q (1+|\ln t|),$$
and direct calculations yield
\begin{align*}
\Big|\frac{-\alpha'(t)-\frac{1-\alpha(t)}{t}}{t^{\alpha(t)}}\Big|=\Big|\frac{-t\alpha'(t)-(1-\alpha(t))}{t^{1+\alpha(t)}}\Big|=\Big|\frac{\int_0^t\int_s^t-\alpha''(y)dy ds}{t^{1+\alpha(t)}}\Big|\leq Q\frac{t^2}{t^{1+\alpha(t)}}\leq Q.
\end{align*}
Combine the above four estimates to get
$$|\bar g'|\leq Q (1+|\ln t|).$$
Thus, we follow exactly the same analysis procedure as that in Section \ref{sec72} to reach the same conclusions as Theorem \ref{thm72} for the case $\alpha_0=1$.
\subsection{Comments on the restriction}
From the above analysis, we find that for either $\alpha_0$ or $\alpha_1$, the restriction $f(0)=0$ is required to ensure the well-posedness, which is not needed for the case that $0<\alpha_0<1$. The inherent reason is that  for either $\alpha_0=0$ or $\alpha_0=1$, the kernel $k$ in the integral equation (\ref{vie}) is indeed a non-singular kernel. Specifically, for $\alpha_0=0$ we could apply 
$$\lim_{t\rightarrow 0^+}|-\alpha(t)\ln t|\leq \lim_{t\rightarrow 0^+}|Qt\ln t|=0$$
 to get
$\lim_{t\rightarrow 0^+}k(t)=1$, and for $\alpha_0=1$ with $\alpha'(0)\neq 0$, (\ref{lim}) implies $\lim_{t\rightarrow 0^+}k(t)=-\alpha'(0)$. As for the case of non-singular kernels, the solutions are usually bounded under smooth data (cf. Theorem \ref{thm72}), one could pass the limit $t\rightarrow 0^+$ in the integral equation (\ref{vie}) to find that $f(0)$ needs to be $0$ to ensure the consistency. 

\noindent\underline{\textbf{Restriction for subdiffusion (\ref{VtFDEs}) with $\alpha_0=1$}}

The above observation could also be extended for the variable-exponent subdiffusion model (\ref{VtFDEs}) with $\alpha_0=1$. By the estimate (\ref{lim}) we have $\lim_{t\rightarrow 0^+} k(t)=-\alpha'(0)$. If $\alpha'(0)\neq 0$, the kernel $k$ is a non-singular kernel. Thus when we pass the limit $t\rightarrow 0^+$ on both sides of (\ref{VtFDEs}), its well-posedness may require the condition
$$-\Delta u(0)=f(0),$$
if $u(\bm x,t)$ is an absolutely continuous function for each $x\in\Omega$. In fact, such condition has been pointed out in, e.g. \cite{Die,StyAML} for nonlocal-in-time diffusion models with non-singular kernels, and what we did above is showing that the kernel $k(t)$ is such a non-singular kernel when $\alpha_0=1$ and $\alpha'(0)\neq 0$. 

\vspace{0.2in}
\noindent\textbf{Conflict of interest statement}\\
On behalf of all authors, the corresponding author states that there is no conflict of interest.

\vspace{0.1in}
\noindent\textbf{Data availability statement}\\
No datasets were generated or analyzed during the current study.

\vspace{0.1in}
\noindent\textbf{Acknowledgments}\\
This work was partially supported by the National Natural Science Foundation of China (No. 12301555), the National Key R\&D Program of China (No. 2023YFA1008903), and the Taishan Scholars Program of Shandong Province (No. tsqn202306083).

\end{document}